\documentclass[letterpaper,12pt,oneside,reqno]{amsart}
\usepackage[T2A]{fontenc}%
\usepackage[utf8]{inputenc}%
\usepackage[english]{babel}%
\usepackage{amsmath,amssymb,amsthm,amsfonts}%
\usepackage{hyperref}%
\usepackage{graphicx}
\usepackage{enumerate}
\usepackage[mathscr]{euscript}
\usepackage{color,tikz}
\usepackage[DIV14]{typearea}
\usepackage[width=.9\textwidth]{caption}
\allowdisplaybreaks%
\numberwithin{equation}{section}%

\usepackage{tikz}
\usetikzlibrary{decorations.markings, arrows}

\newcommand{\Z}{\mathbb{Z}}
\renewcommand{\C}{\mathbb{C}}
\newcommand{\R}{\mathbb{R}}

\renewcommand{\i}{\mathbf{i}}

\newcommand{\al}{\alpha}

\newcommand{\la}{\lambda}

\DeclareMathOperator{\D}{\mathscr{D}}

\newcommand{\s}{\mathrm{Sign}}

\newtheorem{proposition}{Proposition}[section]
\newtheorem{lemma}[proposition]{Lemma}
\newtheorem{corollary}[proposition]{Corollary}
\newtheorem{theorem}[proposition]{Theorem}
\theoremstyle{definition}
\newtheorem{definition}[proposition]{Definition}
\newtheorem{remark}[proposition]{Remark}
\begin{document}

\title{On a family of symmetric rational functions}

\author[A. Borodin]{Alexei Borodin}
\address{Department of Mathematics, 
Massachusetts Institute of Technology,
77 Massachusetts ave.,
Cambridge, MA 02139, USA\newline
Institute for Information Transmission Problems, Bolshoy Karetny per. 19, Moscow, 127994, Russia}
\email{borodin@math.mit.edu}

%\date{\today}

%\begin{center}
%	 \textbf{Preliminary version} 
%\end{center}

\begin{abstract} This paper is about a family of symmetric rational functions that form a one-parameter generalization of the classical Hall-Littlewood polynomials. We introduce two sets of (skew and non-skew) functions that are akin to $P$ and $Q$ Hall-Littlewood polynomials. We establish (a) a combinatorial formula that represents our functions as partition functions for certain path ensembles in the square grid; (b) symmetrization formulas for non-skew functions; (c) identities of Cauchy and Pieri type; (d) explicit formulas for principal specializations; (e) two types of orthogonality relations for non-skew functions. 

Our construction is closely related to the half-infinite volume, finite magnon sector limit of the higher spin six-vertex (or XXZ) model, with both sets of functions representing higher spin six-vertex partition functions and/or transfer-matrices for certain domains.  
\end{abstract}

\maketitle

\setcounter{tocdepth}{2}
\tableofcontents
\setcounter{tocdepth}{2}

\section{Introduction}\label{sc:intro} The Hall-Littlewood symmetric polynomials are very well studied objects that arise naturally in a variety of group theoretic, representation theoretic, and combinatorial contexts; chapters II-V of Macdonald's book \cite{Macdonald1995} contain a detailed description of their origins as well as a rich structural theory. In the simplest instance, the Hall-Littlewood polynomials have the form
\begin{equation}\label{eq:hl}
P_\la(x_1,\dots,x_n)=\mathrm{const}(\la)\cdot \sum_{\sigma\in S_n}\sigma\left(\prod_{1\le i<j\le n} \frac{x_i-q x_j}{x_{i}-x_{j}}\prod_{i=1}^n x_i^{\la_i}\right),
\end{equation}
where the index $\la$ is a finite string of integers $\la_1\ge\la_2\ge\dots\ge 0$, $S_n$ is the symmetric group on $n$ symbols, permutations $\sigma\in S_n$ act on functions in $n$ variables by permuting the variables, and $q\in \C$ is a parameter.\footnote{This parameter is traditionally denoted by $t$, but in the context of the present paper, $q$ happens to be a much more natural notation.}

In recent years, a rational deformation of the Hall-Littlewood polynomials turned to be extremely useful in probability, more exactly, in large time analysis of certain interacting particle systems in (1+1)-dimensions. This deformation is obtained through replacing  $x_i^{\la_i}$ in the above formula by $((\alpha+\beta x_i)/(\gamma+\delta x_i))^{\la_i}$ with $\alpha,\beta,\gamma,\delta\in \C$.

In a pioneering work, Tracy and Widom \cite{TW_ASEP_Fredholm2008}-\cite{TW_total_current2009}  considered the case of $\alpha=\gamma=1$, $\beta\delta=q^{-1}$, and showed that the corresponding functions are eigenfunctions for the generator of the \emph{asymmetric simple exclusion process}, or ASEP for short. In the equivalent context of the infinite volume, finite magnon sector XXZ model, same functions were considered in a much earlier work of Babbitt and Gutkin \cite{BabbittGutkin1990}, \cite{Gutkin2000} (that followed similar but more extensive work of Babbitt and Thomas for the (less general) XXX model \cite{BabbittThomas1977}), but those papers did not contain complete proofs and remained essentially unnoticed. In a very recent work of Borodin-Corwin-Gorin \cite{BCG}, these functions were also utilized for asymptotic analysis of stochastic (spin $\frac 12$) six vertex model in a quadrant. 

A few years after the work of Tracy and Widom, the case of $\alpha=\beta=\gamma=1$, $\delta=0$, was considered by Borodin-Corwin-Petrov-Sasamoto \cite{BCPS1} in connection with the so-called $q$-TASEP and a $q$-Boson particle system (here `T' in `TASEP' stands for `totally', and total asymmetry means that particles in the system are allowed to move in only one direction). 

The fully general case was introduced by Povolotsky \cite{Povolotsky} and developed by Borodin-Corwin-Petrov-Sasamoto \cite{BCPS2} in connection with the so-called $q$-Hahn TASEP and a corresponding zero range process.

Let us also remark that the quantum integrable systems perspective on the Hall-Littlewood polynomials themselves was developed earlier by Van Diejen \cite{vanDiejen2004HL}, and degenerating Hall-Littlewood polynomials  further leads to classical works on the quantum delta Bose gas, see the introduction to \cite{BCPS2} and references therein. 

The essential property of the deformed functions that made them useful for probabilistic analysis (in addition to them being eigenfunctions for generators of interesting interacting particle systems) consisted in completeness and (bi)orthogonality of them viewed as functions of the index $\la$. This made it possible to explicitly construct and in some cases analyze at large times the transition matrices (or Green's functions) for the corresponding Markov  chains. However, from a structural viewpoint, these two properties (orthogonality and being eigenfunctions of a nice difference operator) are merely a tip of the iceberg of a wealth of algebraic and combinatorial facts that are available for Hall-Littlewood polynomials. 

The principal goal of the present work is to develop further the structural properties of the rational deformations of the Hall-Littlewood polynomials. 

If one takes into account harmless renormalizations of functions and variables, the four deformation parameters $\alpha,\beta,\gamma,\delta$ yield a single independent one, and we shall choose it in a specific way and denote it by $s$. The analog of \eqref{eq:hl} then reads
\begin{equation}\label{eq:F}
F_\mu(u_1,\dots,u_M)=\frac{(1-q)^M}{\prod_{i=1}^M (1-su_i)}\,\sum_{\sigma\in S_M}\sigma\left(\prod_{1\le i<j\le M}\frac{u_i-qu_j}{u_i-u_j}\cdot\prod_{i=1}^M \left(\frac{u_i-s}{1-su_i}\right)^{\mu_i}\right),
\end{equation}
where $\mu=(\mu_1\ge \dots\ge \mu_M)\in\Z_{\ge 0}^M$. Up to simple prefactors, these are the functions that have been previously considered, for different values of $s$, in the above referenced papers.\footnote{The prefactor in the right-hand side of \eqref{eq:F} was chosen so that the whole expression can be viewed as a certain partition function, see below. For other purposes, e.g., for Cauchy type identities like \eqref{eq:cauchy-intr} below, it might be more natural to consider normalized functions $F_\mu/F_{(0,\dots,0)}$.}

We introduce several new objects that are closely related to these functions. 

 $\bullet$\quad  We define a `dual' set of functions $G_\nu$ defined as follows:
For $\nu=(\nu_1\ge\dots\ge \nu_n)\in\Z_{\ge 0}^n$ with last $k\ge 0$ coordinates equal to 0, and any $N\ge 0$, we set
\begin{multline}\label{eq:G}
G_\nu(v_1,\dots,v_N)=\frac{(1-q)^N(s^2;q)_n}{(q;q)_{N-n+k}(s^2;q)_k}\\ \times \sum_{\sigma\in S_N}\sigma\left(\prod_{1\le i<j\le N} \frac{v_i-qv_j}{v_i-v_j}\cdot \prod_{i=1}^{n-k} \frac{v_i}{(1-sv_i)(v_i-s)}\left(\frac{v_i-s}{1-sv_i}\right)^{\nu_i}\cdot \prod_{j=n-k+1}^N\frac{1-q^ksv_j}{1-sv_j}\right).
\end{multline}

$\bullet$\quad  We introduce skew functions $F_{\mu/\la}$, $G_{\nu/\la}$, whose special cases with $\la=\varnothing$ or $\la=(0,\dots,0)$ coincide with $F_\mu$ and $G_\nu$, respectively, and show that all these functions can be defined combinatorially, as partition functions for ensembles of paths in the square grid with specific boundary conditions pictured in Figure \ref{fg:paths}. In the Hall-Littlewood limit $s=0$, this turns into the standard combinatorial formula involving summation over semi-standard Young tableaux, cf. \cite[Section III.5]{Macdonald1995}. 

In this combinatorial definition of the $F$- and $G$-functions, the weight of an ensemble of paths is given by the product of vertex weights over all vertices of the grid, and our vertex weights are close relatives of the matrix elements of the higher spin $R$-matrix for the XXZ (or six-vertex) integrable lattice model, with one representation of $U_q(\widehat{sl_2})$ being two-dimensional (spin $\frac 12$) and the other one being a generic Verma module (whose highest weight is related to the parameter $s$). In fact, the $F$-function \eqref{eq:F} can be viewed as the (half)infinite volume, finite magnon sector limit of the eigenfunctions of the higher spin XXZ model with periodic boundary conditions, and the summation over permutations in \eqref{eq:F} is related to the (coordinate or algebraic) Bethe ansatz for this model. 

The connection to the integrable lattice models was essential for us; it provided motivation as well as a broader viewpoint. However, familiarity with such models is not necessary for most statements and proofs of this work, with the exception of Theorem \ref{th:skew-R}.

$\bullet$\quad We employ a version of the Yang-Baxter equation (that is central to the theory of integrable lattice models) to prove several (skew) Cauchy and Pieri type identities involving our $F$- and $G$-functions. The fact that $F_\mu$'s are eigenfunctions of simple difference operators in $\mu$, which made them useful in probabilistic models, may be viewed as a corollary of one of the Pieri type identities. 
To give an example of our identities, the analog of the Cauchy identity has the form (Corollary \ref{cr:cauchy} below):
\begin{equation}\label{eq:cauchy-intr}
\frac{\prod_{i=1}^M(1-su_i)}{(s^2;q)_M}\sum_{\substack{\nu_1\ge\dots\ge \nu_M\ge 0\\ \nu=0^{n_0}1^{n_1}2^{n_2}\cdots}}\prod_{k\ge 0} \frac{(s^2;q)_{n_k}}{(q;q)_{n_k}} \,F_\nu(u_1,\dots,u_M)G_\nu(v_1,\dots,v_N)=\prod_{\substack{1\le i\le M\\1\le j\le N}} \frac{1-qu_iv_j}{1-u_iv_j}\,.
\end{equation}

$\bullet$\quad We also show that the so-called principal specialization into a geometric progression with ratio $q$ of the skew $G$-functions can be viewed as the (half)infinite volume, finite magnon sector transfer-matrix of the higher spin XXZ model with both representations of $U_q(\widehat{sl_2})$ being arbitrary; this is related to the well-known fusion procedure of Kirillov-Reshetikhin \cite{Kirillov-Reshetikhin} for the higher spin XXZ models. 

Thus, our focus in this paper is a one-parameter generalization of the Hall-Littlewood theory. There is another one-parameter generalization of Hall-Littlewood polynomials known as Macdonald polynomials, cf. \cite[Chapter VI]{Macdonald1995}. These two generalizations 
seems to be completely different at the moment, and it is natural to conjecture that there should be a two-parameter lift of the Hall-Littlewood theory that would unite the two. 

One possible direction that could help in finding such a connection is the theory of Hecke algebras. In a recent work, Takeyama \cite{Takeyama1}, \cite{Takeyama2} showed that the $F$-functions \eqref{eq:F} are closely related to certain rational deformations of the affine Hecke algebras. Turning `affine' to `double affine' may lead to a common generalization of the functions in this paper and Macdonald polynomials, but for now this remains out of our reach. 

Our $F$- and $G$-functions have a few degenerations as the parameters $q$ and $s$ tend to certain special values, and some of those appear to be new, cf. Section \ref{sc:degeneration}. The case of inhomogeneous Schur polynomials discussed in Section \ref{ss:ischur} bears a certain similarity to a recent work of Motegi and Sakai \cite{MotegiSakai1}, \cite{MotegiSakai2} on the so-called Grothendieck polynomials, see also Lascoux and Schützenberger \cite{LS}, Lenart \cite{Lenart} for much earlier works on those polynomials, but we were not able to establish a direct connection yet.  

Another recent work that seems to be related to the present one on the level of ideas, but not yet directly, is that of Betea, Wheeler, and Zinn-Justin \cite{BW}, \cite{BWZ}. 

The paper is organized as follows. 

In Section \ref{sc:weights} we introduce vertex weights and establish their connection with $R$-matrices for the higher spin XXZ model. 
In Section \ref{sc:functions} we define the $F$- and $G$-functions as partition functions of certain collections of paths in the square grid, and show that these functions are symmetric in their parameters. 
Section \ref{sc:cauchy} contains (skew) Cauchy and Pieri type identities.
In Section \ref{sc:symmetrization} we prove symmetrization formulas \eqref{eq:F} and \eqref{eq:G}. Section \ref{sc:principal} deals with principal specializations of $F$- and $G$-functions and their connection to fully general higher spin XXZ $R$-matrices and fusion. 
In Section \ref{sc:orthogonality} we discuss orthogonality relations for the $F$-functions \eqref{eq:F} proved earlier in \cite{BCPS2} and their connection to the present work. 
Section \ref{sc:degeneration} contains a brief description of degenerations of $F$- and $G$-functions as parameters $q$ and $s$ tend to special values.

\subsection*{Acknowledgments} I am very grateful to Ivan Corwin, Vadim Gorin, and Leonid Petrov for numerous discussions that were extremely helpful. I am also very grateful to Ole Warnaar for a number of very valuable remarks. The research was partially supported by NSF grant DMS-1056390.

\section{Vertex weights}\label{sc:weights}
We start by fixing two parameters that we denote by $q$ and $s$. They should be viewed as complex numbers with the condition of being generic --- vanishing of certain algebraic expressions in $q$ and $s$ may make some of our statements below meaningless. As a rule, we will not focus on such degenerations. 

\begin{definition}\label{df:weights} For any four-tuple $(i_1,j_1;i_2,j_2)$ of nonnegative integers, define the corresponding \emph{vertex weight} depending on a (generic) complex parameter $u$ as follows: For any $m\ge 0$,
\begin{align}
w_u(m,0,m,0)&=\frac{1-sq^m u}{1-su}, \\
w_u(m,1,m,1)&=\frac{u-sq^m}{1-su}, \\
w_u(m+1,0,m,1)&=\frac{(1-s^2q^m)u}{1-su},\\
w_u(m,1,m+1,0)&=\frac{1-q^{m+1}}{1-su},\label{eq:weight4}
\end{align}
and $w_u(i_1,j_1;i_2,j_2)=0$ for any other values of $i_1,j_1,i_2,j_2\ge 0$. 
\end{definition}

We shall also represent vertices of type $(i_1,j_1;i_2,j_2)$ pictorially as in Figure \ref{fg:vertex}, where $i_1,j_1,i_2,j_2$ denote the number of arrows on South, West, North, and East edges, respectively. 

\begin{figure}
\includegraphics[scale=1]{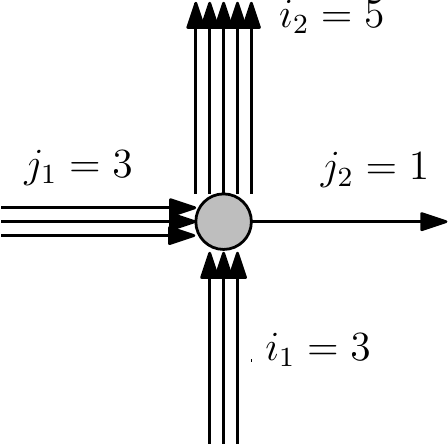}
\caption{Graphical representation of a vertex of type $(i_1,j_1;i_2,j_2)=(3,3;5,1)$.}\label{fg:vertex}
\end{figure}

\begin{remark}\label{rm:arrow_preserve} \textbf{(i)} The set of four-tuples $(i_1,j_1;i_2,j_2)\in\Z_{\ge 0}^4$ whose weights are (generically) nonzero are described by two conditions: $i_1+j_1=i_2+j_2$, and $j_1,j_2\le 1$. The first condition is the `arrow preservation' --- for every vertex with nonzero weight, the number of incoming arrows is equal to the number of outgoing ones. This arrow preservation will be upheld throughout the paper. The second condition says that each horizontal edge carries at most one arrow. This condition will remain relevant until Section \ref{sc:principal}, where it will be lifted, and vertices with arbitrary $(i_1,j_1;i_2,j_2)\in\Z_{\ge 0}^4$ subject to $i_1+j_1=i_2+j_2$ will be allowed to have nonzero weights, cf. Corollary \ref{cr:skew-G}. 

\noindent \textbf{(ii)} The normalization (i.~e., the common denominator $1-su$) is chosen so that $w_u(0,0;0,0)=1$, cf. Remark \ref{rm:partition_function} below. 
\end{remark}
The above-defined vertex weights are closely related to matrix elements of the higher spin $R$-matrix associated with $U_q(\widehat{sl_2})$. To make the connection exact, we need to fix a normalization of the $R$-matrix, and we do so by utilizing the $R$-matrices of \cite{Manga}. There $R_{I,J}$ denotes the image of the universal $R$-matrix in the tensor product of two highest weight representations with arbitrary weights $I$ and $J$ with a particular choice of bases. 

\begin{remark} As was noted in the introduction, familiarity with the theory of integrable lattice models ($R$-matrices etc.) is not really needed for almost all statements and proofs of this paper (with the exception of Theorem \ref{th:skew-R}). Proposition \ref{pr:w-to-R} below details the connection of vertex weights of Definition \ref{df:weights} and $R$-matrices. It is used in our proofs of Proposition \ref{pr:YB} and Theorem \ref{th:skew-R}, but one can also verify the claim of Proposition \ref{pr:YB} in a completely elementary fashion by multiplying $4\times 4$ matrices, as indicated in the beginning of its proof. Thus, a reader could safely omit Proposition \ref{pr:w-to-R} as well as all other mentions of the $R$-matrices without much damage to the content of this work. 
\end{remark}

An explicit formula for $R_{I,J}$ can be seen in (1.1)-(1.3) of \cite{Manga}, where it is assumed that $I\le J$ are nonnegative integers; a more general formula is in \cite[(5.8)-(5.9)]{Manga}. We shall always assume the `field parameter' $\phi$ of \cite{Manga} to be equal to 1, and we shall also re-denote the parameter $q$ from \cite{Manga} by $Q$; it is related to our $q$ above via $Q^2=q$. 

With these conventions we have the following
\begin{proposition}\label{pr:w-to-R} Let $R_{I,J}$ be as in \cite{Manga} with $q$ replaced by $Q$. Then for any $i_1,j_2,i_2,j_2\ge 0$,
\begin{equation}\label{eq:w-to-R}
w_u(i_1,j_1;i_2,j_2)=const\cdot (-1)^{j_1}Q^{\frac{i_2(i_2-1)-i_1(i_1-1)+2i_2+j_1-j_2}2} u^{\frac{j_2-j_1}{2}}\cdot
{\left[R_{I,1}(\la;1)\right]}_{i_1,j_1}^{i_2,j_2},
\end{equation}
where $q=Q^2$, the `spectral' parameter $\la$ is chosen so that $\la^2=(uQ)^{-1}$, and $$const=\frac 1{(1-su)\lambda\, Q^{\frac{1+I}2}}\,.$$ 
\end{proposition}
\begin{proof} A direct comparison of Definition \ref{df:weights} above and \cite[(5.2) and (5.10)]{Manga}. 
\end{proof}
 It is natural to ask why we need a different set of weights here rather than using the $R$-matrix itself. The answer lies in simpler explicit formulas for the symmetric functions involved that are also easier to relate to the Hall-Littlewood polynomials, see Sections \ref{sc:symmetrization} and \ref{sc:degeneration} below.

Our next goal is to describe the Yang-Baxter equation\footnote{The exact meaning of the term `Yang-Baxter equation' may depend on the context in which it is used. Our usage is close to what is called the \emph{star-triangle transformation} in the context of the six vertex model in \cite{Baxter1986}.}  in terms of our vertex weights. 
To that end, define the \emph{two-vertex weights} by 
\begin{equation}\label{eq:two-vertex}
w_{u_1,u_2}^{(m,n)}(k_1,k_2;k_1',k_2')=\sum_{l\ge 0} w_{u_1}(m,k_1;l,k_1')w_{u_2}(l,k_2;n,k_2'),
\end{equation}
where $u_1,u_2\in \C$, $l,m,n\in \Z_{\ge 0}$, $k_1,k_2,k_1',k_2'\in\{0,1\}$. This is the weight of two vertices $(m,k_1;l,k_1')$ and $(l,k_2;n,k_2')$ attached along the $l$-edges with $l\ge 0$ being arbitrary, cf. Figure \ref{fg:two-vertices}. Note that the sum over $l\ge 0$ contains at most one nonzero term, because for both factors to be nonzero we must have 
\begin{equation}\label{eq:arrow_preserve}
l=m+k_1-k_1'=n+k_2'-k_2,
\end{equation}
cf. Remark \ref{rm:arrow_preserve}(i).

\begin{figure}
\includegraphics[scale=1.2]{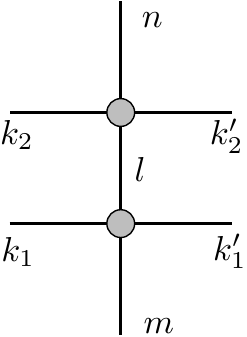}
\caption{Graphical representation of two vertices as in \eqref{eq:two-vertex}.}\label{fg:two-vertices}
\end{figure}

We also set
$$
\widetilde w_{u_1,u_2}^{(m,n)}(k_1,k_2;k_1',k_2')=w_{u_1,u_2}^{(m,n)}(k_2,k_1;k_2',k_1').
$$
As $k_j,k_j'$ vary over $\{0,1\}$, let us organize these weights into $4\times 4$ matrices
$$
w_{u_1,u_2}^{(m,n)}=\begin{bmatrix} 
w_{u_1,u_2}^{(m,n)}(0,0;0,0)&w_{u_1,u_2}^{(m,n)}(0,0;0,1)&w_{u_1,u_2}^{(m,n)}(0,0;1,0)&w_{u_1,u_2}^{(m,n)}(0,0;1,1)\\
w_{u_1,u_2}^{(m,n)}(0,1;0,0)&w_{u_1,u_2}^{(m,n)}(0,1;0,1)&w_{u_1,u_2}^{(m,n)}(0,1;1,0)&w_{u_1,u_2}^{(m,n)}(0,1;1,1)\\
w_{u_1,u_2}^{(m,n)}(1,0;0,0)&w_{u_1,u_2}^{(m,n)}(1,0;0,1)&w_{u_1,u_2}^{(m,n)}(1,0;1,0)&w_{u_1,u_2}^{(m,n)}(1,0;1,1)\\
w_{u_1,u_2}^{(m,n)}(1,1;0,0)&w_{u_1,u_2}^{(m,n)}(1,1;0,1)&w_{u_1,u_2}^{(m,n)}(1,1;1,0)&w_{u_1,u_2}^{(m,n)}(1,1;1,1)
\end{bmatrix},
$$
and similarly for $\widetilde w_{u_1,u_2}^{(m,n)}$.

\begin{proposition}[Yang-Baxter equation]\label{pr:YB} For any $m,n\ge 0$, $u_1,u_2\in \C$, 
\begin{equation}\label{eq:YB-w}
\widetilde w_{u_2,u_1}^{(m,n)}=X\, w_{u_1,u_2}^{(m,n)}\, X^{-1},
\end{equation}
where 
\begin{equation}\label{eq:matrix_X}
X=\begin{bmatrix}  
1&0&0&0\\ \\
0&\dfrac{q(u_1-u_2)}{u_1-qu_2}& \dfrac{(1-q)u_1}{u_1-qu_2}&0\\ \\
0&\dfrac{(1-q)u_2}{u_1-qu_2}&\dfrac{u_1-u_2}{u_1-qu_2}&0\\ \\
0&0&0&1
\end{bmatrix}.
\end{equation}
\end{proposition}
\begin{proof} This relation is easy to check directly, although it may be a bit tedious as one needs to go over the cases $m=n,n\pm 1, n\pm 2$. Let us instead derive it from the Yang-Baxter equation for the $R$-matrices. Its special case that is relevant to us now reads, cf. \cite[(4.8)]{Manga},
\begin{equation}\label{eq:YB-R}
R_{1,I}^{(21)}(\la_1;1)R_{1,I}^{(31)}(\la_2;1)R_{1,1}^{(23)}\left(\frac{\la_2}{\la_1};1\right)=R_{1,1}^{(23)}\left(\frac{\la_2}{\la_1};1\right)R_{1,I}^{(31)}(\la_2;1)R_{1,I}^{(21)}({\la_1};1),
\end{equation}
where the equality takes place in the tensor product of the highest weight representation with weight $I$ that we denote as $V_I$ and $\C^2 \otimes \C^2$, and (omitting zeroes)
\begin{equation}\label{eq:R11}
R_{1,1}(\mu)=\begin{bmatrix}
\mu Q-(\mu Q)^{-1}&&&\\
&\mu-\mu^{-1}&Q-Q^{-1}&\\
&Q-Q^{-1}&\mu-\mu^{-1}&\\
&&&\mu Q-(\mu Q)^{-1}\\
\end{bmatrix}.
\end{equation}
The conjugating matrix $X$ in \eqref{eq:YB-w} is the inverse of $R_{1,1}$ modified according to Proposition \ref{pr:w-to-R}. More exactly, set $u_1=(\la_1^2 Q)^{-1}$, $u_2=(\la_2^2 Q)^{-1}$, and $\mu=\la_2/\la_1$ in \eqref{eq:R11}. Then, using \eqref{eq:arrow_preserve} and its analog for $\widetilde w$, we obtain
\begin{multline*}
w_{u_1,u_2}^{(m,n)}(k_1,k_2;k_1',k_2')=const^2 (-1)^{k_1+k_2} Q^{\frac{n(n-1)-m(m-1)}{2}+2n}
Q^{\frac{(k_1-k_2)-(k_1'-k_2')}{2}}u_1^{\frac{k_1'-k_1}{2}}u_2^\frac{k_2'-k_2}{2}\\ \times \left[R_{1,I}^{(21)}(\la_1;1)R_{1,I}^{(31)}(\la_2;1)\right]_{k_1,k_2}^{k_1',k_2'},
\end{multline*}
\begin{multline*}
\widetilde w_{u_2,u_1}^{(m,n)}(k_1,k_2;k_1',k_2')=const^2 (-1)^{k_1+k_2} Q^{\frac{n(n-1)-m(m-1)}{2}+2n}
Q^{\frac{(k_2-k_1)-(k_2'-k_1')}{2}}u_1^{\frac{k_1'-k_1}{2}}u_2^\frac{k_2'-k_2}{2}\\ \times \left[R_{1,I}^{(31)}(\la_2;1)R_{1,I}^{(21)}({\la_1};1)\right]_{k_1,k_2}^{k_1',k_2'}.
\end{multline*}
The sub- and super-indices $(k_1,k_2)$ and $(k_1',k_2')$ refer to the basis labels in $\C^2\otimes \C^2$, which are ordered as $(0,0),(0,1),(1,0),(1,1)$. We thus obtain
\begin{multline*}
\begin{bmatrix}1&&&\\
&-\left(\frac 1{Qu_2}\right)^\frac12&&\\
&&-\left(\frac Q{u_1}\right)^\frac12&\\
&&&\frac 1{(u_1u_2)^{\frac12}}
\end{bmatrix}^{-1}
w_{u_1,u_2}^{(m,n)}
\begin{bmatrix}
1&&&\\
&-\left( {u_2}Q\right)^\frac12&&\\
&&-\left(\frac{u_1}Q\right)^\frac12&\\
&&&(u_1u_2)^\frac12
\end{bmatrix}^{-1}R_{1,1}\left(\frac{\la_2}{\la_1};1\right)\\
=R_{1,1}\left(\frac{\la_2}{\la_1};1\right)\begin{bmatrix}1&&&\\
&-\left(\frac Q{u_2}\right)^\frac12&&\\
&&-\left(\frac 1{Qu_1}\right)^\frac12&\\
&&&\frac 1{(u_1u_2)^{1/2}}
\end{bmatrix}^{-1}
\widetilde w_{u_2,u_1}^{(m,n)} \begin{bmatrix}
1&&&\\
&-\left(\frac{u_2}Q\right)^{\frac12}&&\\
&&-\left( {Q u_1}\right)^\frac12&\\
&&&(u_1u_2)^{\frac 12}
\end{bmatrix}^{-1}.
\end{multline*}
Simplifying yields \eqref{eq:YB-w}.
\end{proof}

For a future use we record the following computation, cf. \eqref{eq:YB-w}.
\begin{lemma}\label{lm:} Let $A=[A_{ij}]_{i,j=1}^4$ be a $4\times 4$ matrix and $X$ be as in \eqref{eq:matrix_X}. Then
\begin{gather}
(XAX^{-1})_{11}=A_{11},\qquad (XAX^{-1})_{41}=A_{41},\label{eq:lemma1}\\
(XAX^{-1})_{42}=\frac{u_2-u_1}{u_2-qu_1}A_{42}+\frac{(1-q)u_2}{u_2-qu_1}A_{43}.\label{eq:lemma2}  
\end{gather}
\end{lemma}
The proof is straightforward. 

\section{Symmetric rational functions}\label{sc:functions}
The goal of this section is to define certain rational functions in finitely many variables and to show that these functions are symmetric with respect to permutations of the variables. 

We shall define two families of functions, with functions in each family parametrized by pairs of nonnegative signatures. A \emph{signature} is a finite string of ordered integers $\la=(\la_1\ge \la_2\ge \dots \ge \la_L)$.\footnote{I apologize for the use of ``$\la$'' here, given that it was utilized in the previous section to denote the spectral parameter of the $R$-matrix. The notation is traditional in both places, and I find it hard to avoid.} The length $L$ of the string is called the length of the signature. The set of all signatures of a given length $L$ will be denoted as $\s_L$,  
and the set of all \emph{nonnegative} signatures (i.e. the signatures that consist of nonnegative integers) of length $L$ will be denoted as $\s_L^+$. We will also occasionally write nonnegative signatures in the form $\la=0^{m_0}1^{m_1}2^{m_2}\cdots$, where $m_j$'s are the multiplicities: $m_j=|\{i:\la_i=j\}|$. 

We agree that the set $\s_0=\s_0^+$ consists of the single empty signature $\varnothing=0^01^02^0\cdots$. Whenever we speak of the set of all (nonnegative) signatures below, we mean $\bigsqcup_{L\ge 0} \s_L$ or $\bigsqcup_{L\ge 0} \s^+_L$, including the possibility that the length may be zero. 

Let us now consider (a part of) the standard square grid, and let us assign to each vertex of the grid a complex variable in such a way that all the vertices in the same row are assigned the same variable. 

For a finite up-right path in the square grid, we define its weight as the product of weights of its interior vertices\footnote{We call a grid vertex \emph{interior} to a path if it lies on the path and does not coincide with its beginning and ending vertices.}, where the weight of each vertex is determined via Definition \ref{df:weights} with $i_1,j_1,i_2,j_2$ being $0$ or $1$ depending on whether the corresponding (South, West, North, or East) edge adjacent to the vertex is a part of the path or not (0 if it is not, and 1 if it is). To apply the formulas of Definition \ref{df:weights} we assume that $q$ and $s$ are universal parameters fixed once and for all, and for $u$ we utilize the variables that we have just assigned to the vertices of the grid (recall that this variable does not change if we move along any row). 

Similarly, for a collection of finitely many up-right paths in the square grid, we define the corresponding weight as the product of weights of interior vertices of all the paths in the collection, where the weight of each vertex is determined via Definition \ref{df:weights} with $i_1,j_1,i_2,j_2$ equal to the number of paths of the collection that contain the corresponding (South, West, North, or East) edge adjacent to the vertex, with the same convention about $q,s,u$ as in the previous paragraph. Note that the weight of a collection of paths is, generally speaking, not equal to the product of weights of its members (but it will be if the paths do not have common edges).  

We view two collections of up-right paths as identical if their sets of interior vertices coincide and the numbers $(i_1,j_1,i_2,j_2)$ for each of such vertices in the two collections are equal. 

If a collection of paths has no interior vertices we assign to it the weight 1. 

\begin{definition}\label{df:F}
Fix $L\ge M\ge 0$, $\la\in\s^+_L$, $\mu\in\s^+_M$, and indeterminates $u_1,\dots,u_{L-M}$.
Assign to each vertex $(i,j)\in \Z\times \{1,2,\dots,L-M\}$ the variable $u_j$. 
 
Define a rational function $F_{\la/\mu}(u_1,\dots, u_{L-M})$ as the sum of weights of all possible collections of $L$ up-right paths that (cf. Figure \ref{fg:paths}, left panel)
\begin{enumerate}
\item start with the (vertical) edges $\{(\mu_m,0)\to (\mu_m,1),\ 1\le m\le M\}$, and with the (horizontal) edges $\{(-1,j)\to (0,j),\ 1\le j\le L-M\}$;
\item end with the (vertical) edges $\{(\la_l,L-M)\to (\la_l,L-M+1),\ 1\le  l\le L\}$.
\end{enumerate}

We shall also use the abbreviated notation $F_{\la/\varnothing}=F_\la$. 
\end{definition}

\begin{definition}\label{df:G}
Fix $L,k\ge 0$, $\la\in\s_L$, $\nu\in\s_L$, and indeterminates $u_1,\dots,u_k$.
Assign to each vertex $(i,j)\in \Z\times \{1,2,\dots,k\}$ the variable $u_j$. 
 
Define the rational function $G_{\la/\nu}(u_1,\dots, u_k)$ as the sum of weights of all possible collections of $L$ up-right paths that (cf. Figure \ref{fg:paths}, right panel)
\begin{enumerate}
\item start with the (vertical) edges $\{(\nu_n,0)\to (\nu_n,1),\ 1\le n\le L\}$;
\item end with the (vertical) edges $\{(\la_l,k)\to (\la_l,k+1),\ 1\le  l\le L\}$.
\end{enumerate}

We shall also use the abbreviated notation $G_{\la/(0,\dots, 0)}=G_\la$. 
\end{definition}

\begin{figure}
\includegraphics[scale=1.3]{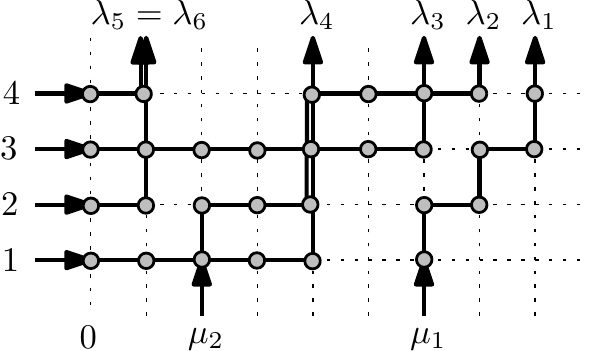}\qquad \includegraphics[scale=1.3]{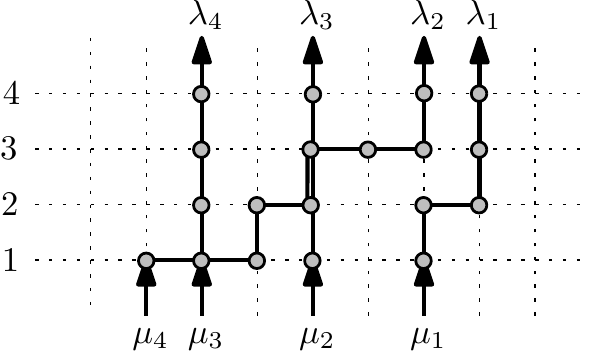}
\caption{Paths for $F_{\la/\mu}$ (left) and $G_{\la/\mu}$ (right).}
\label{fg:paths} 
\end{figure}

Note that in the second definition the signatures are not required to be nonnegative. 

\begin{remark}\label{rm:partition_function}
Because of our normalization $w(0,0;0,0)=1$, we could have equivalently defined the weight of the collection of paths in Definition \ref{df:F} as the product of weights of \emph{all} vertices of the half-strip $\Z_{\ge 0}\times \{1,\dots,L-M\}$. Furthermore, instead of restricting our attention to up-right paths, we could have taken the sum over all possible assignments of nonnegative integers to all the edges adjacent to the vertices of the half-strip that agree with our boundary conditions on the bottom, left, and top boundaries, and such that only finitely many edges carry nonzero numbers (which can be thought of as a boundary condition at the infinite right edge of the strip). Namely to a vertical edge with coordinate $x\in\Z_{\ge 0}$ at the bottom boundary we assign $\mathbf{1}_{x\in \{\mu_m\}}$, to a vertical edge with coordinate $x\in\Z_{\ge 0}$ at the top boundary we assign $\mathbf{1}_{x\in\{\lambda_l\}}$, and to all horizontal edges at the left boundary we assign 1's. 
It is not hard to see that only the assignments that correspond to the collections of up-right paths would give nonzero contributions. 

A similar statement applies for Definition \ref{df:G} as well, with the half-strip replaced by the full strip $\Z\times \{1,\dots,k\}$, and the boundary conditions enforced on top and bottom boundaries (left and right infinities are both taken care of by the finiteness condition). 

\end{remark} 

The definitions immediately imply, by splitting the (half)-strip into two narrower (half)-strips, the following \emph{branching rules}.

\begin{proposition}\label{pr:branching} \textbf{(i)} For any $L\ge K\ge M\ge 0$, $\la\in\s^+_L$, $\mu\in \s^+_M$,
\begin{equation}\label{eq:branching-F}
F_{\la/\mu}(u_1,\dots,u_{L-M})=\sum_{\kappa\in\s^+_K}  F_{\la/\kappa}(u_{K-M+1},\dots,u_{L-M})F_{\kappa/\mu}(u_1,\dots,u_{K-M}). 
\end{equation}
\textbf{(ii)} For any $L,k_1,k_2\ge 0$,  $\la,\nu\in\s_L$, 
\begin{equation}\label{eq:branching-G}
G_{\la/\mu}(u_1,\dots,u_{k_1+k_2})=\sum_{\kappa\in\s_L} G_{\la/\kappa}(u_{k_1+1},\dots,u_{k_1+k_2})G_{\kappa/\nu}(u_1,\dots,u_{k_1}).
\end{equation}
\end{proposition}

The following result is less obvious. It is essentially equivalent to the statement about commutation of transfer matrices with different spectral parameters for the higher spin XXZ model in infinite volume and finite-magnon sector\footnote{The words `finite-magnon sector' refer to the situation when the total number of up-spins in the system remains finite. In our situation this corresponds to finitely many vertical arrows in any row of vertical edges.}. 

\begin{theorem}\label{th:symmetry} The functions $F_{\la/\mu}(u_1,\dots,u_{L-M})$ and $G_{\la/\nu}(u_1,\dots,u_k)$ of Definitions \ref{df:F}, \ref{df:G} are symmetric with respect to permutations of their $u$-variables. 
\end{theorem}
\begin{proof} Due to the branching relations above, it suffices to consider the case of two variables. (In other words, it suffices to show that swapping the variables corresponding to two neighboring rows of the grid does not affect the partition function.)
Also, due to translation invariance of the path collections of Definition \ref{df:G}, for $G_{\la/\nu}$ we may assume that $\la$ and $\nu$ are nonnegative without loss of generality. 

Let us recall the $4\times 4$ matrices $w_{u_1,u_2}^{(m,n)}$ of the two-vertex weights \eqref{eq:two-vertex}. If we consider the product of a few such matrices
$$
w_{u_1,u_2}^{(m_0,\dots,m_S;n_0,\dots,n_S)}=w_{u_1,u_2}^{(m_0,n_0)}w_{u_1,u_2}^{(m_1,n_1)}\cdots w_{u_1,u_2}^{(m_S,n_S)},\qquad S\ge 0, 
$$
then its matrix elements $w_{u_1,u_2}^{(m_0,\dots,m_S;n_0,\dots,n_S)}(k_1,k_2;k_1',k_2')$ can be viewed as sums of products of weights of all vertices in the rectangle $\{0,\dots,S\}\times\{1,2\}$; the summation goes over all possible assignments of nonnegative numbers to the grid edges adjacent to the vertices of the rectangle, subject to boundary conditions given by 
$(m_0,\dots,m_S)$ at the row of vertical edges on the bottom, by $(n_0,\dots,n_S)$ at the row of vertical edges at the top, by $(k_1,k_2)$ at the two horizontal edges on the left boundary, and by $(k_1',k_2')$ at the two horizontal edges on the right boundary, cf. Remark \ref{rm:partition_function}.

Given $\lambda=0^{n_0}1^{n_1}2^{n_2}\cdots$, $\mu=0^{m_0}1^{m_1}2^{m_2}\cdots$ and taking $S\ge \la_1$, we have (by Definition \ref{df:F})
$$
F_{\la/\mu}(u_1,u_2)=w_{u_1,u_2}^{(m_0,\dots,m_S;n_0,\dots,n_S)}(1,1;0,0),
$$ 
or with $\lambda=0^{n_0}1^{n_1}2^{n_2}\cdots$, $\nu=0^{m_0}1^{m_1}2^{m_2}\cdots$ and $S\ge \la_1$ we we have (by Definition \ref{df:G})
$$
G_{\la/\mu}(u_1,u_2)=w_{u_1,u_2}^{(m_0,\dots,m_S;n_0,\dots,n_S)}(0,0;0,0).
$$
Further, Proposition \ref{pr:YB} implies that 
\begin{equation}\label{eq:conjugated-products}
\widetilde w_{u_2,u_1}^{(m_0,n_0)}\widetilde w_{u_2,u_1}^{(m_1,n_1)}\cdots \widetilde w_{u_2,u_1}^{(m_S,n_S)}= X\, w_{u_1,u_2}^{(m_0,n_0)}w_{u_1,u_2}^{(m_1,n_1)}\cdots w_{u_1,u_2}^{(m_S,n_S)}\, X^{-1}
\end{equation}
with $X$ as in \eqref{eq:matrix_X}, and the two relations of \eqref{eq:lemma1} yield
$G_{\la/\nu}(u_1,u_2)=G_{\la/\nu}(u_2,u_1)$ and $F_{\la/\mu}(u_1,u_2)=F_{\la/\mu}(u_2,u_1)$, respectively. 
\end{proof}

\begin{remark}\label{rm:general-symmetry} We used the left boundary conditions $(k_1,k_2)=(1,1)$ and $(0,0)$ to define  $F_*(u_1,u_2)$ and $G_*(u_1,u_2)$. For the symmetry $u_1\leftrightarrow u_2$ it is essential that $k_1$ and $k_2$ are equal. One could define similar rational functions with $k_j$'s being different (either two or more of them for a larger number of variables), but then the symmetry relations would be replaced by more complicated ones; those could be extracted from relations on other matrix elements in the setting of Lemma \ref{lm:}. 
\end{remark}

\section{Identities of Cauchy and Pieri type}\label{sc:cauchy}

In this section we prove several identities involving $F$- and $G$-functions defined above. The terminology we use for these identities (as well as for a few results in further sections as well) originate from the theory of symmetric functions, where it is traditionally used for similar results involving classical Schur symmetric functions and their generalizations. Exact references to analogs of our results for the Hall-Littlewood symmetric functions (which form a one-parameter generalization of the Schur functions, and the $s=0$ specialization of our $F$- and $G$-functions) are collected in Section \ref{ss:HL} below. 

Our first goal is to derive the simplest skew-Cauchy type identity using Proposition \ref{pr:YB} and \eqref{eq:lemma2}. We need more notation to state it. 

\begin{definition}\label{df:conjugated} For any $(i_1,j_1;i_2,j_2)\in \Z_{\ge 0}^4$ we define the \emph{conjugated} {vertex weight} (depending on a complex parameter $u$)
by 
$$
w_u^c(i_1,j_1;i_2,j_2)=\frac{(q;q)_{i_1}(s^2;q)_{i_2}}{(q;q)_{i_2}(s^2;q)_{i_1}}\,w_u(i_1,j_1;i_2,j_2),
$$
with $w_u$ as in Definition \ref{df:weights} and with the standard $q$-Pochhammer notation
$$
(a;q)_n=\begin{cases}
(1-a)(1-qa)\cdots(1-aq^{n-1}),&n\ge 1,\\
1,&n=0.
\end{cases}
$$
Utilizing such conjugated weights instead of the usual ones in Definitions \ref{df:F} and \ref{df:G} leads to the \emph{conjugated} $F$ and $G$ functions
$$
F^c_{\lambda/\mu}:=\frac{c(\lambda)}{c(\mu)}\, F_{\la/\mu},\qquad G^c_{\lambda/\mu}:=\frac{c(\lambda)}{c(\mu)}\,G_{\la/\mu},
$$
where for a signature $\nu=0^{n_0}1^{n_1}2^{n_2}\cdots$ we define
\begin{equation}\label{eq:def-c}
c(\nu)=\prod_{k\ge 0} \frac{(s^2;q)_{n_k}}{(q;q)_{n_k}}\,. 
\end{equation}
\end{definition}

\begin{theorem}[skew-Cauchy identity with single variables]\label{th:skew-cauchy-single} Let $u,v\in\C$ be such that
\begin{equation}\label{eq:for-convergence-single}
\left|\frac{u-s}{1-su}\cdot\frac{v-s}{1-sv}\right|<1. 
\end{equation}
Then for any nonnegative signatures $\la$ and $\mu$ we have
\begin{equation}\label{eq:skew-Cauchy-single}
\sum_\nu F_{\nu/\la}(u)G^c_{\nu/\mu}(v)=\frac{1-quv}{1-uv}\sum_\kappa G^c_{\la/\kappa}(v) F_{\mu/\kappa}(u),
\end{equation}
where both summations are taken over the set of all nonnegative signatures. 
\end{theorem}
\noindent\textbf{Comments.} \textbf{(i)} The summation over $\kappa$ always has finitely many nonzero terms, while the summation over $\nu$ may have infinitely many ones. Condition \eqref{eq:for-convergence-single} is needed to insure the convergence of the series. 

\noindent\textbf{(ii)} For the statement to be nontrivial one must take the length of $\mu$ to be one more than the length of $\la$. Then the only nonzero contributions to the left-hand side of \eqref{eq:skew-Cauchy-single} will come from $\nu$ that are of the same length of $\mu$, and nonzero contributions to the right-hand side will come from $\kappa$ of the same length as $\la$.  

\noindent \textbf{(iii)} The conjugation in \eqref{eq:skew-Cauchy-single} can be placed on the $F$-factors instead of the $G$-factors; the statement obviously does not change. 

\noindent \textbf{(iv)} The formulation and the proof of Theorem \ref{th:skew-cauchy-single} assumes that the `set of all nonnegative signatures' includes the empty signature $\varnothing\in \s_0^+$.

\begin{proof} The argument is similar to the proof of Theorem \ref{th:symmetry}. Namely, we begin with \eqref{eq:conjugated-products} with $\lambda=0^{n_0}1^{n_1}2^{n_2}\cdots$, $\mu=0^{m_0}1^{m_1}2^{m_2}\cdots$, and we also take $u_1=u$, $u_2=v^{-1}$. Further, we look at the matrix element $(4,2)$ of both sides (that corresponds to $(k_1,k_2;k_1',k_2')=(1,1;0,1)$). The type of paths that contribute to the left-hand side can be seen on Figure \ref{fg:skew-cauchy-proof}, left panel. Here for the bottom row of vertices we use parameter $v^{-1}$, and for the top row of vertices we use parameter $u$.

\begin{figure}
\includegraphics[scale=1.2]{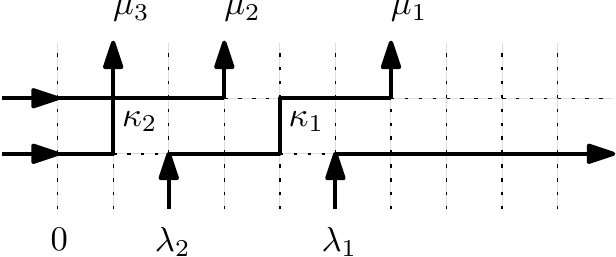}\qquad
\includegraphics[scale=1.2]{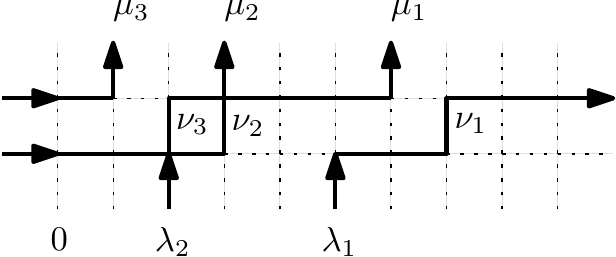}
\caption{Two types of paths in the proof of Theorem \ref{th:skew-cauchy-single}.}
\label{fg:skew-cauchy-proof}
\end{figure}

According to \eqref{eq:lemma2}, on the right-hand side we obtain a linear combination of $(4,2)$ and $(4,3)$ matrix elements of $w_{u,v^{-1}}^{(m_0,n_0)}w_{u,v^{-1}}^{(m_1,n_1)}\cdots w_{u,v^{-1}}^{(m_S,n_S)}$. The $(4,3)$-matrix element is again the sum of weights of paths of the same type as before, but with $u$ and $v^{-1}$ interchanged. On the other hand, the $(4,2)$-matrix element is the sum of weights corresponding to paths of the type pictured in Figure \ref{fg:skew-cauchy-proof}, right panel. The bottom row of vertices uses $u$, and the top row of vertices uses $v^{-1}$.

As $S$ --- the horizontal size of our rectangles --- tends to infinity, all three terms collect a growing number of factors, which are the weights of vertices $(i_1,j_1;i_2,j_2)=(0,1;0,1)$ lying on the long horizontal parts of the paths that exit through the right boundary. 
Let us divide all three terms by 
$$
(w_{v^{-1}}(0,1;0,1))^S=\left(\frac{1-sv}{v-s}\right)^S. 
$$
This will remove most factors from left-hand side and from the (4,2)-matrix element on the right hand-side, with both tending to a finite limit as $S\to\infty$ (we will identify these limits shortly). On the other hand, the (4,3)-element on the right-hand side will equal to a finite expression times 
$$
\left(\frac{w_u(0,1;0,1)}{w_{v^{-1}}(0,1;0,1)}\right)^S=\left(\frac{(u-s)(v-s)}{(1-us)(1-vs)}\right)^S,
$$
which will tend to zero because of our hypothesis \eqref{eq:for-convergence-single}. Hence, from \eqref{eq:lemma2} in the limit $S\to\infty$ we read
\begin{equation}\label{eq:(4,2)-elts}
\text{(4,2)-element of the LHS} = \frac{1-uv}{1-quv}\cdot  \left(\text{(4,2)-element of  } w_{u,v^{-1}}^{(m_0,n_0)}w_{u,v^{-1}}^{(m_1,n_1)}\cdots w_{u,v^{-1}}^{(m_S,n_S)}\right),
\end{equation}
with the two (4,2)-elements represented as sum of weights of paths on the left and right panels of Figure \ref{fg:skew-cauchy-proof}, respectively, where for the top row of vertices on the left figure and for the bottom row of vertices on the right figure we use the weights $w_u$ of Definition \ref{df:weights}, while for the bottom row of vertices on the left figure and for the top row of vertices on the right figure we use the weights
$$
\text{weight}(i_1,j_1;i_2,j_2)=\frac{v-s}{1-sv}\cdot w_{v^{-1}}(i_1,j_1;i_2,j_2). 
$$
Observe that we re-packaged the renormalization as the extra pre-factor $(v-s)/(1-vs)$ on the right, which turns the weight of each of the infinitely many $(0,1;0,1)$ vertices into 1.

It remains to identify \eqref{eq:(4,2)-elts} with \eqref{eq:skew-Cauchy-single}. This readily follows from the pictorial interpretation and the identity
$$
\frac{v-s}{1-sv}\cdot w_{v^{-1}}(i_1,j_1;i_2,j_2)=w_v^c(i_2,1-j_1;i_1,1-j_2). 
$$
Note that pictorially, the change $(j_1,j_2)\mapsto (1-j_1,1-j_2)$ in the above relation correspond to swapping filled and unfilled horizontal edges on the top row of the left panel and on the bottom row of the right panel of Figure \ref{fg:skew-cauchy-proof}. 
\end{proof}

\begin{remark} Similarly to Remark \ref{rm:general-symmetry}, we could have used other matrix elements in the above argument. This leads to different identities. 
More exactly, in the setting of Lemma \ref{lm:} we have
$$
(XAX^{-1})_{22}=\frac{q(u_1-u_2)^2}{(u_1-qu_2)(qu_1-u_2)}\,A_{22}+\frac{(1-q)u_1(u_1-u_2)}{(u_1-qu_2)(qu_1-u_2)}\,A_{32}+\text{ lin. comb. of }(A_{23},A_{33}),
$$
which translates into (under the same assumption \eqref{eq:for-convergence-single})
\begin{multline*}
(1-quv)(q-uv)\sum_\kappa G_{\mu/\kappa}(u)G^c_{\la/\kappa}(v)\\={q(1-uv)^2}\sum_{\nu} G_{\nu/\la}(u)G^c_{\nu/\mu}(v)-{(1-q)uv(1-uv)}\sum_\nu F_{\nu/\la}(u)F_{\nu/\mu}^c(v).
\end{multline*}
Similarly,
$$
(XAX^{-1})_{32}=\frac{(u_1-u_2)^2}{(u_1-qu_2)(qu_1-u_2)}\,A_{32}+\frac{(1-q)u_2(u_1-u_2)}{(u_1-qu_2)(qu_1-u_2)}\,A_{22}+\text{ lin. comb. of }(A_{23},A_{33})
$$
translates into
\begin{multline*}
(1-quv)(q-uv)\sum_\kappa F_{\mu/\kappa}(u)F^c_{\la/\kappa}(v)\\={(1-uv)^2}\sum_{\nu} F_{\nu/\la}(u)F^c_{\nu/\mu}(v)-{(1-q)(1-uv)}\sum_\nu G_{\nu/\la}(u)G_{\nu/\mu}^c(v).
\end{multline*}
\end{remark}

We now draw a few corollaries of Theorem \ref{th:skew-cauchy-single}.

\begin{corollary}[skew-Cauchy identity]\label{cr:skew-cauchy}
Let $u_1,\dots,u_M;v_1,\dots,v_N\in\C$ be such that
\begin{equation}\label{eq:for-cauchy-convergence}
\left|\frac{u_i-s}{1-su_i}\cdot\frac{v_j-s}{1-sv_j}\right|<1,\qquad 1\le i\le M, \quad 1\le j\le N. 
\end{equation}
Then for any nonnegative signatures $\la$ and $\mu$ we have
\begin{equation}\label{eq:skew-cauchy}
\sum_\nu F_{\nu/\la}(u_1,\dots,u_M)G^c_{\nu/\mu}(v_1,\dots,v_N)=\prod_{\substack{1\le i\le M\\1\le j\le N}}\frac{1-qu_iv_j}{1-u_iv_j}\sum_\kappa G^c_{\la/\kappa}(v_1,\dots,v_N) F_{\mu/\kappa}(u_1,\dots,u_M),
\end{equation}
where both summations are over the set of all nonnegative signatures. 
\end{corollary}
\begin{proof} One first uses branching rules of Proposition \ref{pr:branching} (they hold for conjugated $F$ and $G$ functions too), and then applies Theorem \ref{th:skew-cauchy-single} a total of $MN$ times. 
\end{proof}

\begin{corollary}[Pieri type rules]\label{cr:pieri} \textbf{(i)} For any $M\ge 0$,  $\mu\in\s_M^+$, $u_1,\dots,u_M,v\in \C$ such that 
$$
\left|\frac{u_i-s}{1-su_i}\cdot\frac{v-s}{1-sv}\right|<1,\qquad 1\le i\le M, 
$$
we have
\begin{equation}\label{eq:pieri-F}
\prod_{i=1}^M\frac{1-qu_iv}{1-u_iv}\, F_\mu(u_1,\dots,u_M)=\sum_{\nu\in \s_{M}^+}G^c_{\nu/\mu}(v)F_\nu(u_1,\dots,u_M).
\end{equation}
\textbf{(ii)} For any $l,N\ge 0$, $\la\in \s_l^+$, $u,v_1,\dots,v_N\in \C$ such that
$$
\left|\frac{u-s}{1-su}\cdot\frac{v_j-s}{1-sv_j}\right|<1,\qquad 1\le j\le N, 
$$
\begin{equation}\label{eq:pieri-G}
\prod_{j=1}^N\frac{1-quv_j}{1-uv_j}\, G_\la^c(v_1,\dots,v_N)=\frac{1-su}{1-q^{l+1}}\sum_{\nu\in \s_{l+1}^+}F_{\nu/\la}(u)G_\nu^c(v_1,\dots,v_N).
\end{equation}
\end{corollary}
\begin{proof} For (i) we set $\la=\varnothing$ in \eqref{eq:skew-cauchy}. For (ii) we set $\mu=0^{l+1}$ in \eqref{eq:skew-cauchy} and note that $F_{0^{l+1}/\kappa}(u)$ may only be nonzero if $\kappa=0^l$, in which case $$F_{0^{l+1}/0^l}(u)=w_u(l,1;l+1,0)=\frac{1-q^{l+1}}{1-su}$$ according to Definition \ref{df:weights}.
\end{proof}

\begin{remark}\label{rm:bethe} Corollary \ref{cr:pieri}(i) can be viewed as the statement that the vector $$
\{c(\mu)F_\mu(u_1,\dots,u_M)\mid \mu\in\s_M^+\}
$$
is an eigenvector of the `transfer-matrix' 
$$
\{G_{\nu/\mu}(v) \mid \mu,\nu\in \s_M^+\}.
$$
This matrix can indeed be seen as the infinite volume, finite-magnon sector limit of the transfer-matrix of the higher spin XXZ model with periodic boundary conditions (modulo some modifications, cf. Proposition \ref{pr:w-to-R}). As eigenvectors of such transfer-matrices are computable by (coordinate or algebraic) Bethe ansatz, one might expect that there should be a symmetrization formula for $F_\mu$. We shall derive such a formula (and another one for $G_\nu$) in the next section. 
\end{remark}

\begin{corollary}[Cauchy identity]\label{cr:cauchy} For any $M,N\ge 0$, $u_1,\dots,u_M;v_1,\dots,v_N\in \C$ such that \eqref{eq:for-cauchy-convergence} holds, we have
\begin{equation}\label{eq:cauchy}
\frac{\prod_{i=1}^M(1-su_i)}{(q;q)_M}\sum_{\nu\in \s_M^+}F_\nu(u_1,\dots,u_M)G_\nu^c(v_1,\dots,v_N)=\prod_{\substack{1\le i\le M\\1\le j\le N}} \frac{1-qu_iv_j}{1-u_iv_j}\,.
\end{equation}
\end{corollary}
\begin{proof} Substitute $\lambda=\varnothing$, $\mu=0^M$ into \eqref{eq:skew-cauchy}, and using \eqref{eq:weight4} evaluate
$$
F_{0^M}(u_1,\dots,u_M)=(q;q)_M{\prod_{i=1}^M(1-su_i)^{-1}}\,.\qquad\qquad \qedhere
$$
\end{proof}

\section{Symmetrization formulas for $F_\la$ and $G_\la$}\label{sc:symmetrization} The goal of this section is to prove the following statement, cf. Remark \ref{rm:bethe}. In what follows we denote the symmetric group on $n$ symbols by $S_n$, and for $\sigma\in S_n$ and a function $f$ in $n$ variables we also use the notation $\sigma(f)(x_1,\dots,x_n)=f(x_{\sigma(1)},\dots,x_{\sigma(n)}).$ 
\begin{theorem}\label{th:symmetrization} \textbf{(i)} For any $M\ge 0$, $\mu\in \s_M^+$, and $u_1,\dots,u_M\in \C$, we have
\begin{equation}\label{eq:symmetrization-F}
F_\mu(u_1,\dots,u_M)=\frac{(1-q)^M}{\prod_{i=1}^M (1-su_i)}\,\sum_{\sigma\in S_M}\sigma\left(\prod_{1\le i<j\le M}\frac{u_i-qu_j}{u_i-u_j}\cdot\prod_{i=1}^M \left(\frac{u_i-s}{1-su_i}\right)^{\mu_i}\right).
\end{equation}

\textbf{(ii)} Fix $n\ge 0$, $\nu\in\s_n^+$, and assume $k\ge 0$ last coordinates of $\nu$ are zero: $\nu_{n-k+1}=\dots=\nu_n=0$. Then for any $N\ge n-k$ we have
\begin{multline}\label{eq:symmetrization-G}
G_\nu(v_1,\dots,v_N)=\frac{(1-q)^N(s^2;q)_n}{(q;q)_{N-n+k}(s^2;q)_k}\\ \times \sum_{\sigma\in S_N}\sigma\left(\prod_{1\le i<j\le N} \frac{v_i-qv_j}{v_i-v_j}\cdot \prod_{i=1}^{n-k} \frac{v_i}{(1-sv_i)(v_i-s)}\left(\frac{v_i-s}{1-sv_i}\right)^{\nu_i}\cdot \prod_{j=n-k+1}^N\frac{1-q^ksv_j}{1-sv_j}\right).
\end{multline}
\end{theorem}
\begin{remark}\label{rm:symmetrization} \textbf{(i)} One sees directly from Definition \ref{df:F} that increasing all coordinates of $\mu\in\s_M^+$ by the same integer $a\ge 0$, $\mu\mapsto \mu+a^M$, is equivalent to adding $a$ vertices of type $(i_1,j_1;i_2,j_2)=(0,1;0,1)$ to each row of the path collections for $F_\mu$. This yields an extra weight factor:
$$
F_{\mu+a^M}(u_1,\dots,u_M)=\prod_{i=1}^M (w_{u_i}(0,1;0,1))^a F_\mu(u_1,\dots,u_M)=\prod_{i=1}^M\left(\frac{u_i-s}{1-su_i}\right)^a F_\mu(u_1,\dots,u_M). 
$$
This is obviously in agreement with \eqref{eq:symmetrization-F}.

\smallskip

\noindent\textbf{(ii)} Definition \ref{df:F} implies that the number of variables of $F_\mu=F_{\mu/\varnothing}$ must be equal to the length of $\mu$, and this is what we have in \eqref{eq:symmetrization-F}. On the other hand, the number of variables $N$ of $G_\nu(v_1,\dots,u_N)$ can be arbitrary. But if $N$ is smaller than the number of nonzero coordinates of $\nu$, then one easily sees that collections of paths of Definition \ref{df:G} with nonzero weight do not exist, and thus $G_\nu(v_1,\dots,v_N)\equiv 0$. The case of the number of variables being at least as large as the number of nonzero coordinates of $\nu$ is covered by \eqref{eq:symmetrization-G}. 

\smallskip

\noindent\textbf{(iii)} If in \eqref{eq:symmetrization-G} we have $N-n+k>0$ then the summation over $\sigma\in S_N$ can be partially performed explicitly by symmetrizing over indices $(n-k+1,\dots,N)$ first. The resulting formula looks as follows:
\begin{multline}\label{eq:symmetrization-G'}
G_\nu(v_1,\dots,v_N)=\frac{(1-q)^{n-k}(s^2;q)_n}{(s^2;q)_k} 
\sum_{\substack{I\subset\{1,\dots,N\}\\|I|=n-k}}
\prod_{i\in I} \frac{v_i}{(1-sv_i)(v_i-s)}
\cdot\prod_{j\notin I} \frac{1-q^ksv_j}{1-sv_j}
\cdot \prod_{\substack{i\in I\\ j\notin I}}
\frac{v_i-qv_j}{v_i-v_j}
\\
\times
\sum_{\substack{\sigma:\{1,\dots,n-k\}\to I\\ \sigma\text{ is a bijection}}}\sigma\left(\prod_{1\le i<j\le n-k} \frac{v_i-qv_j}{v_i-v_j}\cdot \prod_{i=1}^{n-k} \left(\frac{v_i-s}{1-sv_i}\right)^{\nu_i}\right),
\end{multline}
and we used the symmetrization identity (see \cite[(1.4) in Chapter III]{Macdonald1995})
\begin{equation}\label{eq:symm-identity}
\sum_{\sigma\in S_p}\sigma\left(\prod_{1\le i<j\le p} \frac{z_i-qz_j}{z_i-z_j}\right)=\frac{(q;q)_p}{(1-q)^p}
\end{equation}
along the way.

\smallskip

\noindent\textbf{(iv)} Formula \eqref{eq:symmetrization-G'} immediately implies that the functions $G_\nu$ are \emph{stable} in the sense that adding 0's to the string of their variables does not change them (indeed, the factor $v_i$ forces indices of the zero variables not to be included in the set $I$ in the summation). On the other hand, this fact is also easy to see from Definition \ref{df:G}, as having a zero variable forces the absence of occupied horizontal edges in the corresponding row, and $w_{u=0}(m,0;m,0)=1$ for any $m\ge 0$. This actually proves a more general stability relation: For any signatures $\la,\nu$ we have
\begin{equation}
\label{eq:stable}
G_{\nu/\la}(v_1,\dots,v_N)=G_{\nu/\la}(v_1,\dots,v_N,0).
\end{equation}

\smallskip

\noindent\textbf{(v)} If $\mu\in\s_M^+$ has no zero coordinates then the collections of paths in Definitions \ref{df:F} and \ref{df:G} for $F_\mu(u_1,\dots,u_M)$ and $G_\mu(u_1,\dots,u_M)$ are almost identical apart from the left-most column. More exactly, one has (using (i) above)
\begin{equation}\label{eq:G-via-F}
\begin{aligned}
G_\mu(v_1,\dots,v_M)&=\prod_{i=1}^M w_{v_i}(i,0;i-1,1)\cdot F_{(\mu-1^M)}(v_1,\dots,v_M)\\ &= (s^2;q)_M
\prod_{i=1}^M \frac{v_i}{v_i-s}\cdot F_{\mu}(v_1,\dots,v_M),
\end{aligned}
\end{equation}
where $(\mu-1^M)=(\mu-1,\dots,\mu_M-1)\in \s_M^+$. This relation also immediately follows from \eqref{eq:symmetrization-F} and \eqref{eq:symmetrization-G'}. 

\smallskip

\noindent\textbf{(vi)} The proof of Theorem \ref{th:symmetrization} we give below is a verification rather than a derivation argument, and one might wonder where \eqref{eq:symmetrization-F} and \eqref{eq:symmetrization-G} came from. The symmetrization formula \eqref{eq:symmetrization-F} for $F_\mu$ can be derived with standard (coordinate or algebraic) Bethe ansatz techniques, cf. Remark \ref{rm:bethe}. As for the symmetrization formula \eqref{eq:symmetrization-G} for $G_\nu$, its derivation is given in Proposition \ref{pr:spatial-check} below, and it is based on \eqref{eq:symmetrization-F}, the Cauchy identity \eqref{eq:cauchy}, and the spatial orthogonality of Theorem \ref{th:spatial}. 
\end{remark}
\begin{proof}[Proof of Theorem \ref{th:symmetrization}] We shall use the branching relations of Proposition \ref{pr:branching} and induction on the number of variables. For a single variable, Definitions \ref{df:F} and \ref{df:G} imply
$$
F_{(\mu_1)}(u)=(w_{u}(0,1;0,1))^{\mu_1}w_u(0,1;1,0)=\frac{1-q}{1-su_1}\left(\frac{u-s}{1-su}\right)^{\mu_1},
$$
\begin{multline*}
G_{(\nu_1,0^{n-1})}(v)\\=\begin{cases} w_v(n,0;n-1,1)(w_{v}(0,1;0,1))^{\nu_1-1}w_v(0,1;1,0)=\dfrac{(1-q)(1-s^2q^{n-1})v}{(v-s)(1-sv)}\left(\dfrac{v-s}{1-sv}\right)^{\nu_1},&\nu_1>0,\\
w_{v}(n,0;n,0)=\dfrac{1-sq^nv}{1-sv},&\nu_1=0,
\end{cases}
\end{multline*}
and all three expressions are in agreement with \eqref{eq:symmetrization-F}, \eqref{eq:symmetrization-G}. 

Let us first prove the inductive step for $F_\mu$. The instance of the branching relation \eqref{eq:branching-F} that we need is
\begin{equation}\label{eq:branching-F-single}
F_\mu(u_1,\dots,u_M)=\sum_{\la\in\s_{M-1}^+} F_{\mu/\la}(u_M)F_\la(u_1,\dots,u_{M-1}).
\end{equation}
Split $\mu$ into (nonempty) clusters of equal coordinates
\begin{equation}\label{eq:clusters}
\mu_1=\dots=\mu_{c_1},\quad \mu_{c_1+1}=\dots=\mu_{c_1+c_2},\quad \dots,\quad \mu_{c_1+\dots+c_{m-1}+1}=\dots=\mu_{M},
\end{equation}
where $\{c_j\}_{j=1}^m$ are the cluster sizes. One easily sees from Definition \ref{df:F} that nonzero contributions to the right-hand side of \eqref{eq:branching-F-single} come only from $\la$'s such that
\begin{multline}\label{eq:interlacing-F}
\la_1=\dots=\la_{c_1-1}=\mu_{c_1},\quad \mu_{c_1}\le\la_{c_1}\le \mu_{c_1+1},\\
\la_{c_1+1}=\dots=\la_{c_1+c_2-1}=\mu_{c_1+c_2},\quad \mu_{c_1+c_2}\le \la_{c_1+c_2}\le \mu_{c_1+c_2+1},\ \dots,\\
\mu_{c_1+\dots+c_{m-1}}\le\la_{c_1+\dots+c_{m-1}}\le \mu_{c_1+\dots+c_{m-1}+1},\quad \la_{c_1+\dots+c_{m-1}+1}=\dots=\la_{M-1}=\mu_{M}.
\end{multline}

It is convenient to switch from variables $\{u_i\}$ to their fractional-linear images
\begin{equation}\label{eq:xi}
\xi_i:=\frac{u_i-s}{1-su_i}\,,\qquad i\ge 1. 
\end{equation}
By the induction hypothesis, we know that $F_\la(u_1,\dots,u_{M-1})$ is a linear combination of monomials $\prod_{i}\xi_i^{\la_{\sigma(i)}}$, $\sigma\in S_{M-1}$, with coefficients in $\C(\xi_1,\dots,\xi_{M_1})$ --- the field of rational functions in $\xi_1,\dots,\xi_{M-1}$. As a first step, we want to prove a similar statement for $F_\mu(u_1,\dots,u_M)$. Before doing that, let us make sure that such a representation is unique. 
\begin{lemma}\label{lm:uniqueness} For any $\alpha,\beta\ge 1$, the functions of the form
$$
f_A:\Z_{\ge 0}^\beta\to \C(\xi_1,\dots,\xi_\alpha),\qquad f_A:(p_1,\dots,p_\beta)\mapsto \prod_{i=1}^\alpha \xi_i^{\sum_{j=1}^\beta A_{ij}p_j},
$$
with $A\in \mathrm{Mat}(\alpha\times\beta,\Z)$, are linearly independent over $\C(\xi_1,\dots,\xi_\alpha)$. In other words, if for $\phi_1,\dots,\phi_R\in \C(\xi_1,\dots,\xi_\alpha)$ and pairwise distinct $A^{(1)},\dots,A^{(R)}\in\mathrm{Mat}(\alpha\times\beta,\Z)$ we have
$$
\phi_1f_{A^{(1)}}(p_1,\dots,p_\beta)+\dots+\phi_R f_{A^{(R)}}(p_1,\dots,p_\beta)=0\quad \textrm{for any  }\  p_1,\dots,p_\beta\in \Z_{\ge 0},
$$
then $\phi_1=\dots=\phi_R=0$. 
\end{lemma}
 The linear transformation $(p_1,\dots,p_\beta)\mapsto (p_1+\dots+p_\beta,p_2+\dots+p_\beta,\dots,p_\beta)$ allows one to replace $\Z_{\ge 0}^\beta$ by $\{p_1\ge p_2\ge \dots\ge p_\beta\}\subset\Z^\beta_{\ge 0}$ in the statement of the lemma. It is this version that we actually need. 
Although Lemma \ref{lm:uniqueness} is not far from being a triviality, we supply a proof at the end of this section.

Let us return to showing that $F_\mu(u_1,\dots,u_M)$ is a linear combination of monomials $\prod_i\xi_i^{\mu_{\tau(i)}}$, $\tau\in S_M$, with coefficients in $\C(\xi_1,\dots,\xi_M)$. As $F_\la(u_1,\dots,u_{M-1})$ is a linear combination of monomials of the form $\prod_i\xi_i^{\lambda_{\sigma(i)}}$, let us start with one such monomial in \eqref{eq:branching-F-single} and see what the summation over $\la$'s as in  \eqref{eq:interlacing-F} gives. We shall take $\sigma=\mathrm{id}$, for other $\sigma$'s the argument and the conclusion are similar. 

We want to trace what happens to powers of $\xi_i$'s as we do the summation over $\la$'s. For now we shall ignore rational coefficients that are independent of the values $\la_i$'s and $\mu_i$'s (but they may depend on multiplicities (cluster sizes) in $\la$ and $\mu$). For each $\la_i$ that is free to move in the corresponding interval, cf. \eqref{eq:interlacing-F}, we split its range into three parts --- the left end, the right end, and strictly between the two ends. We then observe that
\begin{itemize}
\item If $\la_i$ is locked then $\xi_i^{\la_i}$ can be simply read as $\xi_i^{\mu_i}$. The same is valid if $\la_i$ is at the right edge of its range. 

\item If $\la_i$ is at the left edge of its (nontrivial) range then $\la_i=\mu_{i+1}$ so that $\xi^{\la_i}=\xi_i^{\mu_{i+1}}$. But in addition to that, the factor $F_{\mu/\la}(u_M)$ in \eqref{eq:branching-F-single} contains $\xi_M^{\mu_{i}-\mu_{i+1}}$ from the vertices of type $(i_1,j_1;i_2,j_2)=(0,1;0,1)$ between $\mu_{i+1}$ and $\mu_i$ in the top ($M$th) row of the corresponding path collection.\footnote{It is actually $\xi_M^{\mu_i-\mu_{i+1}-1}$ as there are $\mu_i-\mu_{i+1}-1$ such vertices. However, since we are ignoring rational coefficients, we can remove the `$-1$' from the exponent.} Thus, together we obtain $\xi_i^{\mu_{i+1}}\xi_M^{\mu_i-\mu_{i+1}}$. Note that the total degree is $\mu_i$. 

\item If $\la_i$ ranges strictly between $\mu_{i+1}$ and $\mu_i$, we need to do the summation
\begin{equation}\label{eq:sum-inside}
\sum_{\la_i:\mu_{i+1}<\la_i<\mu_i} \xi_i^{\la_i} \xi_M^{\mu_i-\la_i-1}=\frac{1}{\xi_i-\xi_M}\,\xi_i^{\mu_i}-\frac{\xi_i}{\xi_M(\xi_i-\xi_M)}\,\xi_i^{\mu_{i+1}}\xi_M^{\mu_i-\mu_{i+1}},
\end{equation}
which is a linear combination of the contributions of the two previous cases. Here in the left-hand side $\xi_M^{\mu_i-\la_i-1}$ comes from $F_{\mu/\la}(u_M)$, where it corresponds to vertices of type $(0,1;0,1)$ at the top row of the path collection between $\la_i$ and $\mu_i$, similarly to the previous case. 

\item We also need to add the factor $(w_{u_M}(0,1;0,1))^{\mu_M}=\xi_M^{\mu_M}$ from the vertices between $0$ and $\mu_M$ at the top row of the path collection. 

\end{itemize}  

We conclude that in all cases any $\xi_i$ with $1\le i\le M-1$ is being raised to the power given by a coordinate of $\mu$, and the total power over all $\xi_i$'s, $1\le i\le M$, is always $\mu_1+\dots+\mu_M$. (Note that we have ignored contributions to $F_{\mu/\la}(u_M)$ of vertices of all types different from $(0,1;0,1)$; such vertices simply add rational coefficients.) Relying on the symmetry of $F_\mu(u_1,\dots,u_M)$ in the $u$-variables, cf. Theorem \ref{th:symmetry}, and on Lemma \ref{lm:uniqueness}, we can now conclude that $F_\mu(u_1,\dots,u_M)$ is the sum of monomials of the form $\prod_i\xi_i^{\mu_{\tau(i)}}$ over $\tau\in S_M$. Furthermore, because of the symmetry, to obtain a formula for $F_\mu(u_1,\dots,u_M)$ it suffices to find the (rational) coefficient of only one such monomial, for example, of $\prod_i \xi_i^{\mu_i}$ that corresponds to $\tau=\mathrm{id}$. This is what we do next. 

Let us focus on a monomial $\prod_i\xi_{\sigma(i)}^{\la_i}$, $\sigma\in S_{M-1}$, with its coefficient in $F_{\la}(u_1,\dots,u_{M-1})$, substitute it into \eqref{eq:branching-F-single} instead of $F_{\la}(u_1,\dots,u_{M-1})$, do the summation over $\la$ subject to \eqref{eq:interlacing-F}, and read off the resulting coefficient of $\prod_i{\xi}_i^{\mu_i}$. 

By going through the same three cases as above for each $\la_i$, we see that there may be a nontrivial contribution to the coefficient of $\prod_i\xi_i^{\mu_i}$ only if $\sigma$ preserves the subsets 
$$
\{1,\dots,c_1\},\quad \{c_1+1,\dots,c_1+c_2\}, \quad \dots,\quad \{c_1+\dots+c_{m-1}+1,\dots,M-1\}
$$
of $\{1,\dots,M-1\}$, where $c_j$'s are the cluster sizes of $\mu$, cf. \eqref{eq:clusters}, and, furthermore, no $\la_i$ can assume the lowest possible value of its range as long as this range is nontrivial. This means that the relevant ranges of different $\la_i$'s do not intersect, and we can perform the summations over each of them independently. 

Let us start with summation over the interval between the first two clusters. From \eqref{eq:symmetrization-F}, the part of $F_\la(u_1,\dots,u_{M-1})$ to be summed has the form, apart from the easy factor $(1-q)^{M-1}/\prod_{i=1}^{M-1} (1-su_i)$,
\begin{multline}\label{eq:to-be-summed}
\prod_{1\le i<j\le M-1}\frac{u_{\sigma(i)}-qu_{\sigma(j)}}{u_{\sigma(i)}-u_{\sigma(j)}}\prod_{i=1}^{M-1} \xi_{\sigma(i)}^{\la_i}=\prod_{c_1+1\le i<j\le M-1}\frac{u_{\sigma(i)}-qu_{\sigma(j)}}{u_{\sigma(i)}-u_{\sigma(j)}}\prod_{\substack{1\le i\le c_1\\ c_1+1\le j\le M-1}}\frac{u_{i}-qu_{j}}{u_{i}-u_{j}}\prod_{i=c_1+1}^{M-1} \xi_{\sigma(i)}^{\la_i}\\ \times
\prod_{1\le i<j\le c_1}\frac{u_{\sigma(i)}-qu_{\sigma(j)}}{u_{\sigma(i)}-u_{\sigma(j)}}\prod_{i=1}^{c_1} \xi_{\sigma(i)}^{\la_i},
\end{multline}
where we used the fact that $\sigma$ preserves $\{1,\dots,c_1\}$. Let us denote the restriction of $\sigma$ to $\{1,\dots,c_1\}$ by $\sigma_1$; we want to sum over $\sigma_1\in S_{c_1}$ as well. Before doing the summation, the above expression needs to be multiplied by the corresponding part of $F_{\mu/\la}(u_M)$ in \eqref{eq:branching-F-single}, which is 
$$
\begin{cases}
w_{u_M}(1,0;0,1)(w_{u_M}(0,1;0,1))^{\mu_{c_1}-\la_{c_1}-1}w_{u_M}(c_1-1,1;c_1,0),&\mu_{c_1+1}<\la_{c_1}<\mu_{c_1},\\
w_{u_M}(c_1,0;c_1;0),&\la_{c_1}=\mu_{c_1}.
\end{cases}
$$
In view of \eqref{eq:interlacing-F}, we have 
$$
\prod_{i=1}^{c_1} \xi_{\sigma(i)}^{\la_i}=\left(\prod_{i=1}^{c_1-1}\xi_{\sigma_1(i)}\right)^{\mu_{c_1}}
\xi_{\sigma_1(c_1)}^{\la_{c_1}}.
$$
Using \eqref{eq:symm-identity} to sum over $\sigma_1$'s with fixed $\sigma_1(c_1)=l$ and substituting explicit weights from Definition \ref{df:weights}, we rewrite the sum over $\sigma_1$ and $\la_{c_1}$ as
\begin{multline}
\prod_{c_1+1\le i<j\le M-1}\frac{u_{\sigma(i)}-qu_{\sigma(j)}}{u_{\sigma(i)}-u_{\sigma(j)}}\prod_{\substack{1\le i\le c_1\\ c_1+1\le j\le M-1}}\frac{u_{i}-qu_{j}}{u_{i}-u_{j}}\prod_{i=c_1+1}^{M-1} \xi_{\sigma(i)}^{\la_i}\cdot \frac{(q;q)_{c_1-1}}{(1-q)^{c_1-1}}\prod_{1\le i\le c_1}\xi_{i}^{\mu_{c_1}}\\ \times
\sum_{l=1}^{c_1}\prod_{\substack{1\le i\le c_1\\ i\ne l}}\frac{u_{i}-qu_{l}}{u_{i}-u_{l}}\left(\frac{(1-s^2)(1-q^{c_1})u_M}{(1-su_M)^2\xi_l^{\mu_{c_1}}}\sum_{\mu_{c_1+1}<\la_{c_1}<\mu_{c_1}}\xi_M^{\mu_{c_1}-\la_{c_1}-1}\xi_l^{\la_{c_1}}+\frac{1-sq^{c_1}u_M}{1-su_M}\right).
\end{multline}
The sum over $\la_1$ is like in \eqref{eq:sum-inside}, and we can omit the 
term that is similar to the second term of the right-hand side of 
\eqref{eq:sum-inside}, because it has $\xi_l^{\mu_{c_1}-\mu_{c_1+1}}$, while we 
need $\xi_l^{\mu_{c_1}}$ to contribute to the coefficient of 
$\prod_i\xi_i^{\mu_i}$. Further,
$$
\frac{(1-s^2)(1-q^{c_1})u_M}{
(1-su_M)^2(\xi_l-\xi_M)}+\frac{1-sq^{c_1}u_M}{1-su_M}=\frac{u_l-q^{c_1}u_M}{u_l-
u_M}\,,
$$
and the final simplification is achieved with the identity
\begin{equation}\label{eq:old-identity}
\sum_{l=1}^{c_1}\prod_{\substack{1\le i\le c_1\\ i\ne 
l}}\frac{u_{i}-qu_{l}}{u_{i}-u_{l}}\cdot \frac{u_l-q^{c_1}u_M}{u_l-
u_M}=\frac{1-q^{c_1}}{1-q}\prod_{i=1}^{c_1} \frac{u_i-qu_M}{u_i-u_M}\,,
\end{equation}
which follows from evaluating $\oint_{\text{around poles at } 
z=u_1,\dots,u_{c_1}} f(z)dz$ with
$$
f(z)=\frac 
1{(q-1)z}\,\frac{z-q^{c_1} u_M}{z-u_M}\prod_{i=1}^{c_1} \frac{qz-u_i}{z-u_i}
$$
in two different ways --- as the sum of residues at $z=u_i$, $1\le i\le c_1$, 
and as the negative sum of residues at $z=0,u_M,\infty$ (the contributions of 
$z=0$ and $z=\infty$ cancel out).\footnote{Identities of this type are rather old. For example, an elliptic generalization of \eqref{eq:old-identity} can be extracted from \cite[No. 400]{TM}. I am very grateful to Ole Warnaar for pointing this out.}

Hence, the summation over $\la_{c_1}$ and $\sigma_1$ yields the expression
$$
\prod_{c_1+1\le i<j\le 
M-1}\frac{u_{\sigma(i)}-qu_{\sigma(j)}}{u_{\sigma(i)}-u_{\sigma(j)}}\prod_{
\substack{1\le i\le c_1\\ c_1+1\le j\le 
M-1}}\frac{u_{i}-qu_{j}}{u_{i}-u_{j}}\prod_{i=c_1+1}^{M-1} 
\xi_{\sigma(i)}^{\la_i}\cdot \frac{(q;q)_{c_1}}{(1-q)^{c_1}}\prod_{1\le i\le 
c_1}\xi_{i}^{\mu_{i}}\prod_{i=1}^{c_1} \frac{u_i-qu_M}{u_i-u_M}\,.
$$
We continue summing in this fashion (the next step is to sum over $\la_{c_2}$ 
and $\sigma_2=\sigma|_{c_1+1,\dots,c_1+c_2}$), and in the end, having summed 
over the whole of $\la$ and $\sigma$ and multiplied by the weight 
$$
(w_{u_M}(0,1;0,1))^{\mu_M}w_{u_M}(c_m-1,1;c_m,0)=\frac{1-q^{c_m}}{1-su_M}\,\xi_M^{\mu_M}
$$ 
of the vertices between $0$ and $\mu_M$ in the $M$th row, 
we obtain
$$
\frac{1-q}{1-su_M}\prod_{i=1}^m  \frac{(q;q)_{c_i}}{(1-q)^{c_i}}
\prod_{1\le a<b\le m}\prod_{\substack{i\in\,\text{$a$th cluster of $\mu$}\\ 
j\in\, \text{$b$th cluster of $\mu$}}}\frac{u_{i}-qu_{j}}{u_{i}-u_{j}} 
\prod_{1\le i\le 
M}\xi_{i}^{\mu_{i}}
$$
Together with the previously omitted factor $(1-q)^{M-1}/\prod_{i=1}^{M-1} 
(1-su_i)$, this gives the correct coefficient of $\prod_{i}\xi_{i}^{\mu_{i}}$ 
in the right-hand side of \eqref{eq:symmetrization-F} (here we need 
\eqref{eq:symm-identity} again). This completes the inductive step for $F_\mu$ 
and the proof of part (i) of Theorem \ref{th:symmetrization}. 

Let us proceed to path (ii) --- the inductive step for $G_\nu$. Recall that the base of the induction, the case of a single variable, was discussed in the beginning of the proof. 

The proof of the inductive step for $G_\nu$ is largely similar to that for $F_\mu$ above. 
Let us comment on the differences. 

One starts with a branching relation, cf. \eqref{eq:branching-G}, \eqref{eq:branching-G}, which says that for $\nu\in\s_n^+$,
$$
G_\nu(v_1,\dots,v_N)=\sum_{\la\in\s_n^+}G_{\nu/\la}(v_N)G_\la(v_1,\dots,v_{N-1}).
$$
It implies that the $\la$-coordinates must interlace with the $\nu$-coordinates. The interlacing is similar to \eqref{eq:interlacing-F}, except in addition the smallest coordinate $\la_n$ may vary between $\nu_n$ and $0$. 

Same inductive arguments as for $F_\mu$ above show that $G_\nu(v_1,\dots,v_N)$ must be a linear combination of monomials $\prod_i \xi_{\tau(i)}^{\nu_i}$, $\tau\in S_N$, with coefficients in $\C (v_1,\dots,v_N)$, where we take, cf. \eqref{eq:xi},
$$
\xi_i=\frac{v_i-s}{1-sv_i},\qquad i\ge 1.
$$ 
Hence, by the symmetry in $v$-variables of Theorem \ref{th:symmetry}, it suffices to evaluate the coefficient of $\prod_i \xi_i^{\nu_i}$. Then one needs to consider two cases: (a) The number of nonzero coordinates in $\nu$ is strictly smaller than the number of $v$-variables, i.e. $n-k<N$; and (b) The number of nonzero coordinates in $\nu$ is equal to the number of $v$-variables , i.e. $n-k=N$. 

For case (a) the computation literally repeats the one we did for the coefficient of $\prod_i\xi_i^{\mu_i}$ in $F_\mu$ except for the very last factor, where one needs to multiply by the weight of vertex 0 in the top row, which for $G_\nu$ is $w_{v_N}(k,0;k;0)=(1-sq^ku_M)/(1-su_M)$, with $k$ being the multiplicity of $0$ in $\nu$, instead of $w_{u_M}(c_m-1,1;c_m,0)$ that would have been there for $F_\mu$. One easily checks that this exactly gives the coefficient of $\prod_i \xi_i^{\nu_i}$ in the right-hand side of \eqref{eq:symmetrization-G} or \eqref{eq:symmetrization-G'}. 

A more substantial difference comes up in case (b). If $n-k=N$, then even though interlacing of $\la$ and $\nu$ (that comes from non-vanishing of $G_{\nu/\la}(v_N)$) allows $\la_N=\la_{n-k}$ to vary between $\nu_{N}$ --- the last nonzero coordinate of $\nu$ ---
and 0, the fact that for $G_\la(v_1,\dots,v_{N-1})$ to be nonzero the number of nonzero coordinates in $\la$ cannot be greater than $N-1$ forces $\la_{N}$ to be zero. If we denote by $c$ the size of the smallest nonzero cluster of $\nu$ (located at $\nu_{N}$), then the total contribution of vertices between $0$ and $\nu_{N}$ to $G_{\nu/\la}(v_N)$
is 
\begin{multline*}
w_{v_N}(k+1,0;k,1)(w_{v_N}(0,1;0,1))^{\nu_{N}-1}w_{v_N}(c-1,1;c,0)=\frac{(1-s^2q^k)v_N}{1-sv_N}\,\xi_N^{\nu_N-1}\,\frac{1-q^c}{1-sv_N}\\ 
=(1-s^2q^k)(1-q^c)\,\frac{v_N}{(1-sv_N)(v_N-s)}\,\xi_N^{\nu_N}.
\end{multline*} 
This allows us to verify that the induction step yields the correct coefficient of $\prod_i \xi_i^{\nu_i}$ in the right-hand side \eqref{eq:symmetrization-G} and concludes the proof of Theorem \ref{th:symmetrization}.
\end{proof}

\begin{proof}[Proof of Lemma \ref{lm:uniqueness}] Assume that we found nonzero $\phi_1,\dots,\phi_R\in\C(\xi_1,\dots,\xi_\alpha)$ and pairwise distinct $A^{(1)},\dots,A^{(R)}\in\mathrm{Mat}(\alpha\times\beta,\Z)$ such that $\sum_{r}\phi_r f_{A^{(r)}}\equiv 0$. Pick an $\alpha$-tuple of positive reals $(\zeta_1,\dots,\zeta_\al)\in \R^\al_{>0}$ and a direction $(\omega_1,\dots,\omega_\beta)\in\Z^\beta_{>0}$ so that the numbers
$$
\Omega_r =\sum_{i,j} A_{ij}^{(r)}\zeta_i \omega_j,\qquad 1\le r\le R, 
$$
are pairwise distinct (this is always possible as an equality of two such numbers is a nontrivial quadratic equation on $\zeta$'s and $\omega$'s, and solutions to finitely many such equations cannot exhaust
$\R_{>0}^\alpha\times \Z_{>0}^\beta$). Assign weights to variables $\xi_i$ via $\mathrm{wt}(\xi)=\zeta_i$, $1\le i\le \al$, and single out top homogeneous components of the polynomial numerators and denominators of the rational functions $\phi_1,\dots,\phi_R$ with respect to this weighting. Pick a point $(c_1,\dots,c_\alpha)\in (\C\setminus\{0\})^\alpha$ so that none of these top homogeneous components vanishes at this point. 

Let us now take a fixed large integer $L$ and look at the behavior of $\sum_{r=1}^R \phi_r f_{A^{(r)}}$ as we substitute 
$$
(\xi_1,\dots,\xi_\alpha)=(c_1\xi^{\zeta_1},\dots,c_\alpha\xi^{\zeta_\alpha}),\qquad 
(p_1,\dots,p_\beta)=(\omega_1L,\dots,\omega_\beta L)
$$ 
and take $\xi\to\infty$. We observe that each $|\phi_r|$ behaves as a nonzero constant times $|\xi|^{const_r}$. Each $|f_{A^{(r)}}|$ equals a nonzero constant times $|\xi|^{\Omega_rL}$. As long as $L\cdot \min_{1\le r\le R}(\max_{1\le r\le R}\Omega_r-\Omega_i)$ is greater than all $const_r$ coming from $\phi_r$ (which we can guarantee by taking $L$ large enough), the term corresponding to the maximal $\Omega_r$ will dominate all the other ones, and hence $\sum_r \phi_r f_{A^{(r)}}$ cannot vanish. The contradiction completes the proof of Lemma \ref{lm:uniqueness}.  
\end{proof}

\section{Principal specializations}\label{sc:principal}

In this section we provide explicit formulas for $F_\mu$, $G_\nu$, and $G_{\nu/\la}$ specialized at geometric progressions with ratio $q$. While the first two results are elementary corollaries of the symmetrization formulas in Theorem \ref{th:symmetrization}, the third one is less obvious as it relies on fusion rules for transfer matrices of the higher spin XXZ model. 

\begin{proposition}\label{pr:principal} \textbf{(i)} For any $M\ge 0$, $\mu\in\s_M^+$, and $u\in \C$, we have
\begin{equation}\label{eq:principal-F}
F_\mu(u,qu,\dots,q^{M-1}u)=\frac{(q;q)_M}{(su;q)_M}\prod_{i=1}^{M}\left(\frac{q^{i-1}u-s}{1-sq^{i-1}u}\right)^{\mu_{i}}.
\end{equation}

\textbf{(ii)} Fix $n\ge 0$, $\nu\in\s_n^+$, and assume $k\ge 0$ coordinates of $\nu$ are zero. Then for any $N\ge n-k$ and $v\in\C$ we have
\begin{equation}
\label{eq:principal-G}
G_\nu(v,qv,\dots,q^{N-1}v)=\frac{(q;q)_N(s^2;q)_n(sv;q)_{N+k}}{(q;q)_{N-n+k}(s^2;q)_k(sv;q)_n(sv;q)_N(sv^{-1};q^{-1})_{n-k}}\,\prod_{i=1}^{n-k}\left(\frac{q^{i-1}v-s}{1-sq^{i-1}v}\right)^{\nu_{i}}.
\end{equation} 
\end{proposition}
\begin{proof} For (i) we use \eqref{eq:symmetrization-F}. Observe that substituting $\{u_i=uq^{i-1}\}_{i=1}^M$ into 
$$
\sigma\left(\prod_{1\le i<j\le M} \frac{u_i-qu_j}{u_i-u_j}\right),\qquad \sigma\in S_M,
$$
gives 0 unless $\sigma=\mathrm{id}$, in which case we get $(q;q)_M/(1-q)^M$. This implies \eqref{eq:principal-F}. In the same way \eqref{eq:symmetrization-G} implies \eqref{eq:principal-G}. 
\end{proof}

Let us proceed to the skew functions $G_{\la/\mu}$. 
\begin{definition} \label{df:corresponds}
We say that a function $H(\la,\mu)$ of two signatures $\la$ and $\mu$ of the same length corresponds to vertex weights $\{w^{(H)}(i_1,j_1;i_2,j_2)\mid i_1,j_1,i_2,j_2\ge 0\}$ if each value $H(\la,\mu)$ is given by the sum of products of these vertex weights over all possible collections of upright paths as in Definition \ref{df:G} and Figure \ref{fg:paths}, right panel, with a single horizontal row of vertices (i.e. $k=1$ in the notation of Definition \ref{df:G}).
\end{definition}

Clearly, with the terminology of Definition \ref{df:corresponds}, the single variable specialization $G_{\la/\mu}(v)$ corresponds to the weights $w_v$ of Definition \ref{df:weights}. This is the case when Definition \ref{df:G} and \ref{df:corresponds} simply coincide. 

Because of the unfortunate overload of the letter $\la$, cf. the footnote in the beginning of Section \ref{sc:functions}, in what follows we speak about functions of $\nu$ and $\mu$, where $\nu$ plays the role of $\la$ in Definition \ref{df:corresponds}. 

\begin{theorem}\label{th:skew-R} For any $J\ge 1 $, $G_{\nu/\mu}(v,qv,\dots,q^{J-1}v)$ corresponds (as a function of $\nu$ and $\mu$, in the sense of Definition \ref{df:corresponds}) to the vertex weights
\begin{equation}\label{eq:w-via-R}
w_v^{(J)}(i_1,j_1;i_2,j_2)=\frac{(-1)^{j_1}Q^{\frac{(2i_2-I-1)J+i_2^2-i_1^2}{2}}}{\la^J(sv;q)_J}\cdot {\left[R_{I,J}(\la;1)\right]}_{i_1,j_1}^{i_2,j_2},
\end{equation}
where $R_{I,J}$ is the (fully general) higher spin $R$-matrix of the XXZ model as in \cite[(5.8)-(5.9)]{Manga} with the parameter $q$ of \cite{Manga} re-denoted by $Q$ and related to our $q$ through $Q^2=q$, the spectral parameter $\la$ given by $\la^2=(Q^{J}v)^{-1}$, parameter $I$ given by $Q^{-I}=s$, and $m(I,J)$ in \cite[(5.9)]{Manga} (used to denoted the minimum of $I$ and $J$) set to $J$. 
\end{theorem}

\begin{remark}\label{rm:skew-R} \textbf{(i)} For $J=1$, Theorem \ref{th:skew-R} immediately follows from Proposition \ref{pr:w-to-R}. Comparing the two formulas one may notice discrepancy in certain factors of the form $f(j_1)/f(j_2)$; this is explained by the fact that such factors cancel out in products of weights over vertices of up-right paths, and are thus irrelevant (for the purpose of Theorem \ref{th:skew-R}). 

\smallskip

\noindent\textbf{(ii)} As we know from Definition \ref{df:weights}, for $J=1$ the vertex weights vanish as long as $\max(j_1,j_2)>1$. Similarly, $w_v^{(J)}(i_1,j_1;i_2,j_2)$ from \eqref{eq:w-via-R} vanishes if $\max(j_1,j_2)>J$; this corresponds to the highest weight representation of $U_q(\widehat{sl_2})$ with weight $J$ having dimension $J+1$. 

Also, if $I\in \{1,2,\dots\}$, i.e. $s^2\in\{q^{-1}, q^{-2},\dots\}$, then $w_v^{(J)}(i_1,j_1;i_2,j_2)=0$ as long as $\max(i_1,i_2)>I$. In particular, $s^2=q^{-1}$ corresponds to the spin-$\frac 12$ situation with no more than one particle (vertical arrow) at each location. 
\end{remark}

\begin{proof}[Proof of Theorem \ref{th:skew-R}] The argument combines Proposition \ref{pr:w-to-R} and an infinite volume limit of the fusion relation \cite[(7.13)]{Manga} (the fusion relations in this context were first derived in \cite{Kirillov-Reshetikhin}, but it is convenient for us to use the notation of \cite{Manga}). 

More exactly, follow \cite{Manga} in combining the $R$-matrices into transfer-matrices via
\begin{equation}\label{eq:trace}
T_{J,I}(\la)=\mathrm{Trace}_{V_J} \left[R^{(01)}_{J,I}(\la)\otimes \cdots\otimes R^{(0S)}_{J,I}(\la)\right],
\end{equation}
where $V_J$ is the 0th `auxiliary' highest weight representation of $U_q(\widehat{sl_2})$ with weight $J$, and the tensor product is taken in the `quantum space' $V_I^{\otimes M}$, $M\ge 1$. The fusion relation we are interested in reads, see \cite[(7.13)-(7.14)]{Manga},
\begin{equation}\label{eq:fusion}
T_{1,I}(\la)T_{J,I}(\la Q^{-\frac{J+1}{2}})=
\left( (Q\la-(Q\la)^{-1})(Q^{-I}\la-Q^I\la^{-1}) \right)^M 
T_{J-1,I}(\la Q^{-\frac{J}2 -1})+
T_{J+1,I}(\la Q^{-\frac{J}2}),
\end{equation}
where we take $J<I$. We want to take the limit $M\to\infty$ while keeping the number of quantum particles finite (i.e., looking at matrix elements of transfer-matrices with finitely many indices that are different from 0). To do that, we first normalize $R_{J,I}$ so that $[R_{J,I}(\la,1)]_{0,0}^{0,0}$ turns into 1; this is achieved by 
\begin{equation}\label{eq:normalize-R}
R_{J,I}(\la,1)\mapsto \widetilde R_{J,I}(\la,1):=\frac{Q^{-\frac{I(J+1)}{2}}}{\la^J(\lambda^{-2}Q^{-I-J};Q^2)_J} \,R_{J,I}(\la,1).
\end{equation}
Note that this does not destroy the second coefficient `1' in the right-hand side of \eqref{eq:fusion}. Further, let us assume that $Q$ is sufficiently close to 1, and $\la$ is sufficiently close to $s=Q^{-I}$; as the final formula \eqref{eq:w-via-R} easily admits analytic continuation from such a domain (it is even an identity of rational functions for fixed $\nu$ and $\mu$), this is not a serious restriction. This leads to the first coefficient in the right-hand side of \eqref{eq:fusion} being a small constant raised to the power $M$, and the fusion relation now makes sense in the $M\to \infty$ limit; it reads
$$
\widetilde T_{1,I}(\la)\widetilde T_{J,I}(\la Q^{-\frac{J+1}{2}})=
\widetilde T_{J+1,I}(\la Q^{-\frac{J}2}),
$$
or, iterating,
\begin{equation}\label{eq:transfer}
\widetilde T_{1,I}(Q^{\frac{-J+1}{2}}\la)\widetilde T_{1,I}(Q^{\frac{-J+3}{2}}\la)\cdots \widetilde T_{1,I}(Q^{\frac{J+1}{2}}\la)=\widetilde T_{J,I}(\la),
\end{equation}
where we use the tilde in $\widetilde T$ to signify both the $M\to \infty$ limit and the normalization \eqref{eq:normalize-R}. 

From the point of view of path interpretation of matrix elements of the transfer matrix, we started with the periodic boundary conditions in \eqref{eq:trace}, and by making $\la$ close to $s$ we made $[R_{1,I}(\la,1)]_{j_1=1,i_1=0}^{j_2=1,i_2=0}$ (which corresponds to $w_{v}(0,1;0,1)$) small, which in the limit $M\to\infty$ prevents paths from going around the infinite loop. 

It only remains to utilize Proposition \ref{pr:w-to-R} in the left-hand side, keeping in mind the normalization \eqref{eq:normalize-R} and the symmetry relation \cite[(5.10)]{Manga}.  
\end{proof}

Theorem \ref{th:skew-R} and results of \cite{Manga} allow us to write down an explicit formula for the principal specialization of the skew $G$-functions. For that we need additional notation. 

Following \cite{Manga}, we shall use the following extended notation for $q$-Pochhammer symbols:
$$
(x;q)_n=\begin{cases} (1-x)(1-qx)\cdots(1-q^{n-1}x),&n>0,\\
1,&n=0,\\
\left((1-q^nx)(1-q^{n+1}x)\cdots (1-q^{-1}x)\right)^{-1},&n<0,
\end{cases}
$$
and also regularized terminating basic hypergeometric series
\begin{multline*}
{}_{r+1}\bar{\phi}_r\left(\begin{matrix} q^{-n};a_1,\dots,a_r\\b_1,\dots,b_r\end{matrix}
\Bigl| q,z\right)=\sum_{k=0}^n z^k\,\frac{(q^{-n};q)_k}{(q;q)_k}	\prod_{i=1}^r (a_i;q)_k(b_iq^k;q)_{n-k}\\
=\prod_{i=1}^r (b_i;q)_n\cdot {}_{r+1}{\phi}_r\left(\begin{matrix} q^{-n};a_1,\dots,a_r\\b_1,\dots,b_r\end{matrix}
\Bigl| q,z\right) .
\end{multline*}

\begin{corollary}\label{cr:skew-G} For any $J\ge 1 $, $G_{\nu/\mu}(v,qv,\dots,q^{J-1}v)$ corresponds (as a function of $\nu$ and $\mu$, in the sense of Definition \ref{df:corresponds}) to the vertex weights given by 
\begin{multline}\label{eq:w-via-phi}
\widetilde w_v^{(J)}(i_1,j_1;i_2,j_2)\\=\frac{(-1)^{i_1+j_1} q^{\frac{i_1^2+i_2^2}4+\frac{i_1(j_1-1)+i_2j_2}{2}}s^{j_1-i_1}v^{i_1}(vs^{-1};q)_{j_1-i_2}}{(q;q)_{i_1} (vs;q)_{i_1+j_1}}\, {}_{4}\bar{\phi}_3\left(\begin{matrix} q^{-i_1};q^{-i_2},q^J sv,qsv^{-1}\\s^2,q^{1+j_1-i_2},q^{1+J-i_1-j_1}\end{matrix}
\Bigl|\, q,q\right)
\end{multline}
if $i_1+j_1=i_2+j_2$, and by 0 otherwise. 
\end{corollary}
\begin{proof} This statement is a mere substitution of \cite[(5.8)-(5.9)]{Manga} into Theorem \ref{th:skew-R} and subsequent removal of certain factors of the form $f(j_1)/f(j_2)$, cf. Remark \ref{rm:skew-R}(i). 
\end{proof}

Proposition \ref{pr:principal}(ii) (and hence (i) too via \eqref{eq:G-via-F}) can, in principle, be derived from Corollary \ref{cr:skew-G}, because for $G_{\nu/0^M}(v,qv,\dots,q^{J-1}v)$ only vertices with either $i_1=0$ or $j_1=0$ participate. For $i_1=0$, the ${}_4\bar{\phi}_3$ in \eqref{eq:w-via-phi} is simply equal to 1, and for $j_1=0$ only one term in the series expansion of ${}_4\bar{\phi}_3$ contributes and gives an elementary expression. Multiplying the weights over the row of vertices should yield \eqref{eq:principal-G}, but the computation is rather tedious.

Observe that the right-hand side of \eqref{eq:w-via-phi} is manifestly a polynomial in $q^J$, while the definition of $G_{\nu/\mu}(v,qv,\dots,q^{J-1}v)$ requires $J$ to be a positive integer. One might wonder if $G_{\nu/\mu}(v,qv,\dots,q^{J-1}v)$ can be analytically continued in $q^J$ in a natural way. One answer is provided by the following statement. 

\begin{corollary}\label{cr:cauchy-principal} For any $M\ge 0$, $\mu\in \s_M^+$, and $v,q^J\in \C$, \emph{define} $G_{\nu/\mu}(v,qv,\dots,q^{J-1}v)$ via Corollary \ref{cr:skew-G}. Then for any  $u_1,\dots,u_M\in \C$ in a small enough neighborhood of $s$, we have
\begin{equation}\label{eq:pieri-principal}
\prod_{i=1}^M\frac{1-q^J u_iv}{1-u_iv}\, F_\mu(u_1,\dots,u_M)=\sum_{\nu\in \s_{M}^+}\frac{c(\nu)}{c(\mu)}\, G_{\nu/\mu}(v,qv,\dots,q^{J-1}v)F_\nu(u_1,\dots,u_M),
\end{equation}
where the function $c(\,\cdot\,)$ is defined in \eqref{eq:def-c}. 
\end{corollary}
\noindent\textbf{Comments.} \textbf{(i)} One way to think about \eqref{eq:pieri-principal} is as of a decomposition of the left-hand side, viewed as a function of $\mu\in \s_M$, in the basis of functions $\{F_{\nu}(u_1,\dots,u_M)\}_{\nu\in \s_M}$ on $\s_M$. Because of certain orthogonality relations for $F_\nu$ that we describe in the next section, the coefficients in such an expansion can be effectively extracted via contour integrals. This provides an alternative expression for $G_{\nu/\mu}(v,qv,\dots,q^{J-1}v)$ that is manifestly polynomial in $q^J$, cf. Remark \ref{rm:spatial-check} below. Proving the equality between this expression and that of Corollary \ref{cr:cauchy-principal} directly seems to be challenging though. 

\smallskip

\noindent\textbf{(ii)} One can also view \eqref{eq:pieri-principal} as an eigenrelation for the `fused transfer-matrix' $$[G_{\nu/\mu}(v,qv,\dots,q^{J-1}v)]_{\nu,\mu\in\s_M}$$ and its eigenvector $\{c(\nu)F_{\nu}(u_1,\dots,u_M)\}_{\nu\in \s_M}$, cf. Remark \ref{rm:bethe}. 
The product in the left-hand side of \eqref{eq:pieri-principal} is the corresponding eigenvalue, and it is indeed an infinite volume limit of an eigenvalue of the fused transfer matrix of the higher spin XXZ model.

\noindent\textbf{(iii)} For $\mu=0^M$, $q^J=(vs)^{-1}$, Corollary \ref{cr:cauchy-principal} yields \cite[Proposition 5.18]{BCPS2}. Note that for such value of $q^J$, $\nu$ must have no nonzero coordinates because otherwise the factor $(sv;q)_{N+k}=(q^{-J};q)_{J+k}$ in the right-hand side of \eqref{eq:principal-G} vanishes, and then for $G_\nu(v,qv,\dots,q^{J-1}v)$ we can use \eqref{eq:G-via-F} and \eqref{eq:principal-F}. I am very grateful to Leonid Petrov for pointing this connection out.

\begin{proof}[Proof of Corollary \ref{cr:cauchy-principal}] We argue by analytic continuation of the skew Cauchy identity \eqref{eq:skew-cauchy} with $N=J$, $(v_1,\dots,v_N)=(v,qv,\dots,q^{J-1}v)$, $\lambda=\varnothing$, viewed as an identity between polynomials in $q^J$ for $q^J\in \{q,q^2,q^3,\dots\}$. 

Indeed, one readily sees from Corollary  \ref{cr:skew-G} that $G_{\nu/\mu}(v,qv,\dots,q^{J-1}v)$ is a polynomial in $q^J$ of degree at most $M$, and that its absolute value grows at most as $const^{\nu_1}$ and $\nu_1\to\infty$ as long as $q^J$ stays in a compact subset of $\C$. On the other hand, one sees from \eqref{eq:symmetrization-F} that by taking $u_1,\dots,u_M$ into a small enough neighborhood of $s$ one can achieve that $|F_\nu(u_1,\dots,u_M)|$ decays, as $\nu_1\to\infty$, faster than any (small) positive constant to the power $\nu_1$. Hence, the series in the right-hand side of \eqref{eq:pieri-principal} is uniformly convergent for bounded $q^J$ and $u_i$'s close to $s$, the sum remains a polynomial in $q^J$ of degree at most $M$, and it can be analytically continued off any $M+1$ distinct points. Since the skew Cauchy identity implies \eqref{eq:pieri-principal} for $q^J$ in the infinite set $\{q,q^2,\dots,\}$, the proof is complete.   
\end{proof}

While the expression in the right-hand side of \eqref{eq:w-via-phi} does not look too appetizing, it may simplify for special values of parameters. Here is an example of such a simplification at $v=s$. 

\begin{proposition}\label{pr:q-hahn}
For any $J\ge 1$, and by analytic continuation of Corollaries \ref{cr:skew-G}-\ref{cr:cauchy-principal}, for any $q^J\in\C$, $G_{\nu/\mu}(s,qs,\dots,q^{J-1}s)$ corresponds (as a function of $\nu$ and $\mu$, in the sense of Definition \ref{df:corresponds}) to the vertex weights given by 
\begin{equation}\label{eq:q-hahn}
{\widehat {w}}^{(J)}_s(s,qs,\dots,q^{J-1}s)=(-s)^{-j_1}(s^2q^J)^{j_1}\,\frac{(q^{-J};q)_{j_1}(s^2q^J;q)_{i_2-j_1}}{(s^2;q)_{j_2}}\,\frac{(q,q)_{i_2}}{(q,q)_{j_1}(q;q)_{i_2-j_1}}
\end{equation}
for $i_2\ge j_1$ and $i_1+j_1=i_2+j_2$, and by 0 otherwise. 
\end{proposition}
\begin{remark}\label{rm:q-hahn} \textbf{(i)} The tilde and hat over `$w$' in the left-hand sides of \eqref{eq:w-via-phi} and \eqref{eq:q-hahn} symbolize that the expressions on the right-hand sides are different from \eqref{eq:w-via-R} by (irrelevant) factors of the form $f(j_1)/f(j_2)$, cf. Remark \ref{rm:skew-R}(i), that we remove to make the resulting expressions simpler. 

\smallskip 

\noindent\textbf{(ii)} The condition $i_2\ge j_1$ or, equivalently (modulo the default condition $i_1+j_1=i_2+j_2$), $i_1\ge j_2$, pictorially means that the number of paths that exit any vertex on the top is at least as large as the number of paths that enter the vertex from the left. This could be thought of as the condition that the up-right paths are not allowed to move horizontally by more than one unit. This restriction can also be seen in a different way: At $v=s$, the eigenvalue in the left-hand side of \eqref{eq:pieri-principal} has the form
$$
\prod_{i=1}^M \frac{1-q^Ju_is}{1-u_is}=\prod_{i=1}^M\left(\frac{1-q^Js^2}{1-s^2}+\frac{(q^J-1)s}{1-s^2}\cdot\frac{s-u_i}{1-u_is}\right),
$$
and the important part is that the factors are linear functions in $\xi_i=(s-u_i)/(1-u_is)$. When one multiplies such an expression by \eqref{eq:symmetrization-F}, it is natural to expect that powers $\nu_i$ of $\xi_j$'s increase by no more than one. 

\smallskip

\noindent\textbf{(iii)} The right-hand side of \eqref{eq:q-hahn} coincides, up to the factor $(-s)^{j_1}$, with the jumping probabilities \cite[(8)]{Povolotsky} with the parameters $(\mu,\nu)$ of \cite{Povolotsky} related to ours via $\mu=s^2 q^J$, $\nu=s^2$ (I apologize again for the overloading of Greek letters). Further, Corollary \ref{cr:cauchy-principal} coincides with the eigenrelation first proved in \cite{Povolotsky}, and also re-stated and re-proved as \cite[Proposition 5.13]{BCPS2}. 

The extra factor $(-s)^{j_1}$ is explained by a change of variables in the eigenfunctions: If we replace our variables $u_i$ in $F_\nu(u_1,\dots,u_M)$ by $sz_i$, $1\le i\le M$, then
$$
\frac{u_i-s}{1-su_i}=(-s)\frac{1-z_i}{1-s^2z_i}, \qquad 1\le z_i\le M,
$$
and when we raise this expression to power $\nu_{\sigma(i)}$ and take the product over $i$ as in \eqref{eq:symmetrization-F}, we obtain the extra factor of $(-s)^{\sum_i \nu_i}$. On the other hand, extra $(-s)^{j_1}$ in the weight of any vertex of type $(i_1,j_1;i_2,j_2)$
leads to the multiplication of $G_{\nu/\mu}$ corresponding to these vertex weights by $(-s)^{\sum_i(\mu_i-\nu_i)}$. We thus see that all these powers of $(-s)$ cancel out in the  eigenrelation \eqref{eq:pieri-principal}.
\end{remark}
\begin{proof}[Proof of Proposition \ref{pr:q-hahn}] In principle, we simply need to substitute $v=s$ into the right-hand side of \eqref{eq:w-via-phi}, but what we literally read is not very illuminating. One way to proceed is to apply the transformation formula \cite[(B.3)]{Manga} that reads
$$
{}_{4}\bar{\phi}_3\left(\begin{matrix} q^{-m};a,b,c\\q^{1-m+n},e,f\end{matrix}
\Bigl|\, q,q\right)=\frac{(-1)^{m+n}(ab)^n(a,b,c;q)_{m-n}}{q^{n+\frac{(m-n)(m-n-1)}{2}}}
{}_{4}\bar{\phi}_3\left(\begin{matrix} q^{-n};\frac qa,\frac qb,cq^{m-n}\\q^{1-n+m},\frac{qe}{ab},\frac{qf}{ab}\end{matrix}\Bigl|\, q,q\right)
$$
with $m,n\in\Z_{\ge 0}$ and $abc=ef q^n$ (we used the abbreviated notation $(a,b,c;q)_l=(a;q)_l(b;q)_l(c;q)_l$ here). Setting $v=s$, choosing
$$
m=i_1,\ n=j_2,\ a=s^2q^J,\ b=q,\ c=q^{-i_2},\ e=s^2,\ f=q^{1+J-i_2-j_2},
$$
noting that $q/b=1$ and 
$$
{}_{4}\bar{\phi}_3\left(\begin{matrix} q^{-N};A,1,C\\D,E,F\end{matrix}
\Bigl|\, q,z\right)=(D,E,F;q)_N,
$$
we rewrite the right-hand side of \eqref{eq:w-via-phi} with $v=s$, remembering that $i_1+j_1=i_2+j_2$, as
\begin{gather*}
\frac{(-1)^{i_1+j_1} q^{\frac{i_1^2+i_2^2}4+\frac{i_1(j_1-1)+i_2j_2}{2}}s^{j_1}(1;q)_{j_1-i_2}}{(q;q)_{i_1} (s^2;q)_{i_2+j_2}}\\
\times
\frac{(-1)^{i_1+j_2}(s^2q^{J+1})^{j_2}(s^2q^J,q,q^{-i_2};q)_{i_2-j_1}}{q^{j_2+\frac{(i_1-j_2)(i_1-j_2-1)}{2}}}\cdot(q^{1+i_1-j_2},q^{-J},q^{1-i_2-j_2}s^{-2};q)_{j_2}.
\end{gather*}
A few factors can now be simplified:
\begin{gather*}
(1;q)_{j_1-i_2}=\mathbf{1}_{i_2\ge j_1} \cdot (-1)^{i_2+j_1}\, \frac{q^\frac{(i_2-j_1)(i_2-j_1+1)}{2}}{(q;q)_{i_2-j_1}}\,;\\
\frac{(q^{1-i_2-j_2}s^{-2};q)_{j_2}}{(s^2;q)_{i_2+j_2}}=\frac{(-1)^{j_2}s^{-2j_2}q^{-\frac{(2i_2+j_2-1)j_2}{2}}}{(s^2;q)_{i_2}}\,;\\
\frac{(q^{-i_2};q)_{i_2-j_1}(q^{1+i_1-j_2};q)_{j_2}}{(q;q)_{i_1}}=\frac{(-1)^{i_2+j_1}q^{-\frac{(i_2+j_1+1)(i_2-j_1)}{2}}(q;q)_{i_2}}{(q;q)_{j_1}(q;q)_{i_2-j_1}}\,.
\end{gather*}
Collecting powers of $(-1)$ gives $(-1)^{j_1}$. Remembering that we can multiply vertex weights by $f(j_1)/f(j_2)$, cf. Remarks \ref{rm:skew-R}(i) and \ref{rm:q-hahn}(i), we can replace
$$
(s^2q^J)^{j_2}(q^{-J};q)_{j_2}\ \text{  by  }\  (s^2q^J)^{j_1} (q^{-J};q)_{j_1} \qquad \text{and}\qquad  s^{j_1}s^{-2j_2}\text{  by  } s^{-j_1}.
$$
It remains to collect the powers of $q$. We read
$$
\frac{i_1^2+i_2^2}{4}+\frac{i_1(j_1-1)+i_2j_2}{2}+\frac{(i_2-j_1)(i_2-j_1+1)}{2}-\frac{(2i_2+j_2-1)j_2}{2}-\frac{(i_2+j_1+1)(i_2-j_1)}{2}.
$$
Substituting $i_1=i_2+j_2-j_1$ into this expression yields $\frac14(j_1^2-j_2^2)$, which can be removed from the vertex weight, because $q$ raised to that power has the form $f(j_1)/f(j_2)$. Gathering remaining factors leads to the desired expression.
\end{proof}

\section{Orthogonality relations}\label{sc:orthogonality} The functions $F_\nu(u_1,\dots,u_n)$ satisfy two types of (bi)orthogonality relations, both were proved in \cite{BCPS2} (one of the relations had been previously conjectured in \cite{Povolotsky}), and we shall restate them below. The goal of the section is to explain how results of the previous sections connect to these orthogonality relations. 

We shall show that the Cauchy identity of Corollary \ref{cr:cauchy} essentially implies one of them, the so-called \emph{spectral biorthogonality}. We shall also show that the second one, the so-called \emph{spatial biorthogonality}, can be viewed as a link between the two symmetrization formulas of Theorem \ref{th:symmetrization}. This theorem can thus be used to give a proof of the spatial biorthogonality, but we stop short of doing that because the proof given in \cite{BCPS2} is shorter and more direct. On the other hand, one can also say that the spatial biorthogonality provides a route of \emph{deriving} (rather than verifying) the symmetrization formula for $G_\nu$ given the simpler, Bethe ansatz type symmetrization formula for $F_\mu$, and this is indeed how the formula for $G_\nu$ was obtained for the first time. 

The orthogonality relations can be complemented with \emph{completeness} of the corresponding functional bases, thus providing Plan\-che\-rel type isomorphism theorems between suitable functional spaces. Such theorems were the primary goal of \cite{BCPS2}, where an interested reader can find their statements and proofs. 

\begin{theorem}[Spectral orthogonality; Theorems 4.3, 6.7 of \cite{BCPS2}]\label{th:spectral} For any $n\ge 1$, we have
\begin{multline}
\label{eq:spectral-1}\sum_{\nu\in\s_n}\oint\cdots\oint\prod_{i=1}^n\frac{du_i}{2\pi\i} \oint\cdots\oint\prod_{i=1}^n\frac{dv_i}{2\pi\i}\sum_{\nu\in\s_n} c(\nu)F_{\nu}(u_1,\dots,u_n)F_\nu(v_1^{-1},\dots, v_n^{-1})\\ \times \prod_{1\le i<j\le n}(u_i-u_j)(v_i-v_j)\varphi(u_1,\dots,u_n)\psi(v_1,\dots,v_n)\\ 
=(-1)^{\frac{n(n-1)}{2}}\oint\cdots\oint\prod_{i=1}^n\frac{du_i}{2\pi\i}\prod_{1\le i, j\le n} (u_i-qu_j)\sum_{\sigma\in S_n}\varphi(u_1,\dots,u_n)\psi(u_{\sigma(1)},\dots,u_{\sigma(n)}),
\end{multline}
where $u_i$'s and $v_j$'s are integrated over positively oriented circles $|u_i-s|=|v_i-s|=\epsilon\ll 1$, $\varphi,\psi$ are suitable test functions, and $c(\,\cdot\,)$ is given by \eqref{eq:def-c}. Less formally, the above relation can be rewritten as
\begin{equation}
\begin{gathered}\label{eq:spectral-2}
\prod_{1\le i<j\le n}(u_i-u_j)(v_i-v_j)\sum_{\nu\in\s_n} c(\nu)F_{\nu}(u_1,\dots,u_n)F_\nu(v_1^{-1},\dots, v_n^{-1}) \\
=(-1)^{\frac{n(n-1)}{2}}\prod_{1\le i, j\le n} (u_i-qu_j)\cdot\det\left[\delta(v_i-u_j^{-1})\right]_{i,j=1}^n. 
\end{gathered}
\end{equation}
\end{theorem}
 The statement of Theorem \ref{th:spectral} is somewhat sloppy with not defining exactly the class of test functions for which \eqref{eq:spectral-1} holds. One possible class is described in \cite{BCPS2}, another one will come out of our proof of Theorem \ref{th:spectral} below. Neither of them is optimal, and since our goals are mostly algebraic here, we do not pursue this issue further. 

\begin{theorem}[Spatial orthogonality; Corollary 3.13 of \cite{BCPS2}]\label{th:spatial} For any $n\ge 1$, $\mu,\nu\in\s_n$, we have
\begin{equation}\label{eq:spatial}
\frac{c(\nu)}{(q-1)^nn!} \oint\cdots\oint\prod_{i=1}^n\frac{du_i}{2\pi\i u_i}\prod_{1\le i\ne j\le n}\frac{u_i-u_j}{u_i-qu_j}\,F_\nu(u_1,\dots,u_n)F_\mu(u_1^{-1},\dots,u_n^{-1})=\mathbf 1_{\mu=\nu},
\end{equation}
where $u_i$'s are integrated over a positively oriented contour that contains points $\{s,qs,\dots, q^{n-1}s\}$ and its own image under multiplication by $q$, and that does not contain $s^{-1}$. 
\end{theorem}

Let us proceed to showing how the Cauchy identity implies spectral orthogonality. The argument is similar to the proof of \cite[Proposition 6.10]{BCPS1}, where it was used for the Hall-Littlewood symmetric polynomials. 
\begin{proof}[Proof of Theorem \ref{th:spectral}] We start with the Cauchy identity \eqref{eq:cauchy} with $M=N=n$, substitute $v_i\mapsto v_i^{-1}$, and rewrite it using \eqref{eq:G-via-F}:
\begin{multline}\label{eq:cauchy-rewrite}
\frac{\prod_{i=1}^n(1-su_i)}{(q;q)_n}\sum_{\nu:\,\nu_n=0} F_\nu(u_1,\dots,u_n) G_\nu^c(v_1^{-1},\dots, v_n^{-1})\\+
{\prod_{i=1}^n(1-su_i)(1-sv_i)}{}\sum_{\nu:\nu_n\ge 1}{c(\nu)} F_\nu(u_1,\dots,u_n)F_\nu(v_1^{-1},\dots,v_n^{-1})=\prod_{i,j=1}^n \frac{v_i-qu_j}{v_i-u_j}\,.
\end{multline}

It is convenient to use the notation
$$
\xi_i=\frac{u_i-s}{1-su_i},\quad u_i=\frac{\xi_i+s}{1+s\xi_i},\qquad \zeta_i=\frac{v_i-s}{1-sv_i},\quad v_i=\frac{\zeta_i+s}{1+s\zeta_i},\qquad 1\le i\le n. 
$$
Let us multiply both sides of \eqref{eq:cauchy-rewrite} by
\begin{equation}\label{eq:multiply-by}
\prod_{i=1}^n \frac{\zeta_i^M}{\xi_i^M}\prod_{1\le i<j\le n}(u_i-u_j)(v_i-v_j)\cdot  \varphi(\xi_1,\dots,\xi_n)\varphi(\zeta_1,\dots,\zeta_n),
\end{equation}
where $M$ is a large positive integer, and $\varphi(\xi_1,\dots,\xi_n),\psi(\zeta_1,\dots,\zeta_n)$ are test functions such that 
$$
\prod_{i=1}^n(1+s\xi_i)^{-n-1}\varphi(\xi_1,\dots,\xi_n)\in\C[\xi^{\pm 1},\dots,\xi_n^{\pm 1}],
$$ 
and $\psi(\zeta_1,\dots,\zeta_n)$ is analytic in a punctured neighborhood of the origin with a possible finite order pole at the origin (as was noted above, these conditions on test functions are not optimal can be substantially relaxed, but we do not pursue that here). Let us further integrate over $|u_i-s|=\epsilon\ll 1$, $|v_i-s|=2\epsilon$. 

Rewriting the $u$-integrals in terms of $\xi$-variables, we observe that our assumptions and \eqref{eq:symmetrization-F} guarantee that in each term of the left-hand side of \eqref{eq:cauchy-rewrite} we are integrating a Laurent polynomial in $\xi_i$'s. Furthermore, in the first sum with $\nu_n=0$, each term will have at least one $\xi_i$ raised to only very negative degrees --- this follows from the symmetrization identity \eqref{eq:symmetrization-F} and the presence of $\prod_i \xi_i^{-M}$ with $M\gg 1$. This implies that the sum with $\nu_n=0$ vanishes after the $u$-integration. 

To deal with the sum with $\nu_n\ge 1$ we observe that, cf. Remark \ref{rm:symmetrization}(i),
\begin{equation}
\prod_{i=1}^n\frac{\zeta_i^M}{\xi_i^M}\sum_{\nu:\nu_n\ge 1} c(\nu)F_\nu(u_1,\dots,u_n)F_\nu(v_1^{-1},\dots,v_n^{-1})=\sum_{\nu:\nu_n\ge -M+1} c(\nu)F_\nu(u_1,\dots,u_n)F_\nu(v_1^{-1},\dots,v_n^{-1}).
\end{equation}
In the limit $M\to \infty$ we thus obtain the sum over all $\nu\in\s_n$. (With our test functions, this limit is actually a stabilization as terms with small or large enough $\nu_j$'s give zero contribution by the same reasoning that we used to remove the part with $\nu_n=0$ above.)

It remains to understand what the integral of the right-hand side of \eqref{eq:cauchy-rewrite} multiplied by \eqref{eq:multiply-by} gives. For that we shrink the $v$-contours to $s$. Due to $\prod_i \zeta_i^M$, $M\gg 1$, there is no singularity at $v_i=s$ for any $i$, $1\le i\le n$. Hence, for a nonzero contribution one needs to pick the residues at the first order poles $v_i=u_j$. Two different $v_i$'s cannot utilize that same $u_j$ because of the $\prod_{i<j}(v_i-v_j)$ factor in \eqref{eq:multiply-by}. Therefore, we end up with the sum 
$$
\sum_{\sigma\in S_n} \mathrm{Res}_{v_i=u_{\sigma(i)}}(\cdots),
$$
and this yields the right-hand side of \eqref{eq:spectral-1}. 
\end{proof}

Let us proceed to the spatial orthogonality \eqref{eq:spatial}. We are going to test it on the Cauchy identity \eqref{eq:cauchy}. More exactly, this identity provides a decomposition of any function of $u_1,\dots,u_M$ that has the form as in the right-hand side of \eqref{eq:cauchy} divided by the prefactor of the left-hand side, on $\{F_\nu(u_1,\dots,u_M)\}_{\nu\in s_M^+}$. We are going to extract the coefficients in such decomposition using \eqref{eq:spatial} and see that they are given by \eqref{eq:symmetrization-G'}:

\begin{proposition}\label{pr:spatial-check} For any $n,N\ge 0$, $\nu\in\s_n^+$, $v_1,\dots, v_N\in \C$ in a sufficiently small neighborhood of $s$, the expression
\begin{equation}\label{eq:spatial-check}
\frac{(q;q)_nc(0^n)}{(q-1)^nn!} \oint\cdots\oint\prod_{i=1}^n\frac{du_i}{2\pi\i u_i}\prod_{1\le i\ne j\le n}\frac{u_i-u_j}{u_i-qu_j}
\prod_{1\le i\le n}\left(\frac{1}{1-su_i}\prod_{\substack{1\le j\le N}}\frac{1-qu_iv_j}{1-u_iv_j}\right)  F_\nu(u_1^{-1},\dots,u_n^{-1})
\end{equation}
with integration contours as in Theorem \ref{th:spatial},
coincides with the right-hand side of \eqref{eq:symmetrization-G'} with $k$ being the number of zero coordinates in $\nu$. 
\end{proposition}
\begin{remark}\label{rm:spatial-check} \textbf{(i)} The above statement coupled with Theorem \ref{th:symmetrization}(ii) can be used to provide a proof of Theorem \ref{th:spatial}. To do that one needs to argue that functions $\{F_\nu(u_1,\dots,u_n)\}_{\nu\in\s_n}$ are linearly independent so that decompositions on them yield well-defined coefficients, and that the set of functions 
$$
\prod_{1\le i\le n}\left(\frac{1}{1-su_i}\left(\frac{u_i-s}{1-su_i}\right)^a\prod_{\substack{1\le j\le N}}\frac{1-qu_iv_j}{1-u_iv_j}\right), \qquad
a\in\Z,\ N\in \Z_{\ge 0},\ v_1,\dots,v_N\in\C,
$$
is sufficiently rich (the extra parameter $a$ appeared due to Remark \ref{rm:symmetrization}(i)). We do not pursue this direction here as the proof of Theorem \ref{th:spatial} given in \cite{BCPS2} is simpler. 

\smallskip

\noindent\textbf{(ii)} One can similarly apply spatial orthogonality \eqref{eq:spatial} to the skew Cauchy identity \eqref{eq:pieri-F} with $\la=\varnothing$, thus obtaining a contour integral formula for the skew functions $G_{\nu/\mu}(v_1,\dots,v_N)$ with arbitrary $\mu,\nu\in \s_n$ and $v_1,\dots,v_N\in \C$. Via Theorem \ref{th:skew-R}, this can be used to obtain explicit expressions for the fully general higher spin $R$-matrix for the XXZ model. 
\end{remark}

\begin{proof}[Proof of Proposition \ref{pr:spatial-check}] All parts of the integrand in \eqref{eq:spatial-check} are symmetric in $u_i$'s. This implies that we can de-symmetrize $F_\mu(u_1^{-1},\dots,u_n^{-1})$ using Theorem \ref{th:symmetrization}(i) and equivalently rewrite \eqref{eq:spatial-check} as (also recalling that $c(0^n)=(s^2;q)_n/(q;q)_n$)
\begin{multline}\label{eq:spatial-desymm}
{(s^2;q)_n} \oint\cdots\oint\prod_{i=1}^n\frac{du_i}{2\pi\i}\prod_{1\le i< j\le n}\frac{u_i-u_j}{u_i-qu_j}\\ \times
\prod_{1\le i\le n}\left(\frac{1}{(u_i-s)(1-su_i)}\left(\frac{1-su_i}{u_i-s}\right)^{\nu_i}\prod_{\substack{1\le j\le N}}\frac{1-qu_iv_j}{1-u_iv_j}\right). 
\end{multline}

Let us first handle integration over the last $k$ variables (recall that $k$ is the number 
of zero coordinates in $\nu$). If $k\ge 1$, then inside the $u_n$-integration contour, the integrand viewed as a function in $u_n$ has only one simple pole at $u_n=s$. (Indeed, by our assumption on the contours, the points $s^{-1}, v_1^{-1}, \dots, v_N^{-1}$ lie outside the $u_n$-contour.) Evaluating the residue gives, from all the factors that involve $u_n$,
$$
\frac{1}{1-s^2}\prod_{i=1}^{n-1}\frac{u_i-s}{u_i-qs}\prod_{j=1}^N \frac{1-qsv_j}{1-sv_j}\,.
$$
Assuming $k\ge 2$, i.e. $\nu_{n-1}=0$, we see that in $u_{n-1}$ the pole at $u_{n-1}=s$ gets canceled by the last expression, which has also introduced a simple pole at $u_{n-1}=qs$. As for $u_n$, there are no other singularities inside the $u_{n-1}$-integration contour, and we can evaluate the residue at $u_{n-1}=qs$, which gives, writing only factors that depend on $u_{n-1}$, 
$$
\frac{1}{1-qs^2}\prod_{i=1}^{n-2}\frac{u_i-qs}{u_i-q^2s}\prod_{j=1}^N \frac{1-q^2sv_j}{1-qsv_j}\,.
$$
As we continue in this fashion, after $k$ steps our integral \eqref{eq:spatial-desymm} turns into
\begin{multline}\label{eq:spatial-without-zeroes}
\frac{(s^2;q)_n}{(s^2;q)_k} \oint\cdots\oint\prod_{i=1}^{n-k}\frac{du_i{(u_i-s)}}{2\pi\i(u_i-q^ks)}\prod_{j=1}^N\frac{1-q^ksv_j}{1-sv_j}\prod_{1\le i< j\le n-k}\frac{u_i-u_j}{u_i-qu_j}\\ \times
\prod_{1\le i\le n-k}\left(\frac{1}{(u_i-s)(1-su_i)}\left(\frac{1-su_i}{u_i-s}\right)^{\nu_i}\prod_{\substack{1\le j\le N}}\frac{1-qu_iv_j}{1-u_iv_j}\right). 
\end{multline}

All remaining $\nu_1\ge\dots\ge \nu_{n-k}$ are at least 1, and we cannot proceed in the same way. However, this means that there are no singularities at $u_i=s^{-1}$. Also, there is $\sim |u_i|^{-2}$ decay near $u_i=\infty$, which means that if we peel off the contours and deform them to $\infty$, we only need to take into account the poles at $u_i=v_j^{-1}$. No two $u_i$'s are allowed to share the same pole because of the factor $\prod_{i<j}(u_i-u_j)$ in the integrand. Evaluating (negative) residues at all possible pole choices yields
\begin{multline*}
\frac{(s^2;q)_n}{(s^2;q)_k} \sum_{\substack{\sigma:\{1,\dots,n-k\}\to\{1,\dots,N\}\\ \sigma \text{ is injective}}}\prod_{i=1}^{n-k}\frac{{v_{\sigma(i)}^{-1}-s}}{v_{\sigma(i)}^{-1}-q^ks}\prod_{j=1}^N\frac{1-q^ksv_j}{1-sv_j}\prod_{1\le i< j\le n-k}\frac{v_{\sigma(i)}^{-1}-v_{\sigma(j)}^{-1}}{v_{\sigma(i)}^{-1}-qv_{\sigma(j)}^{-1}}\\ \times
\prod_{1\le i\le n-k}\left(\frac{1}{(v_{\sigma(i)}^{-1}-s)(1-sv_{\sigma(i)}^{-1})}\left(\frac{1-sv_{\sigma(i)}^{-1}}{v_{\sigma(i)}^{-1}-s}\right)^{\nu_i}(1-q)v_{\sigma(i)}^{-1}\prod_{\substack{1\le j\le N\\ j\ne \sigma(i)}}\frac{1-qv_{\sigma(i)}^{-1}v_j}{1-v_{\sigma(i)}^{-1}v_j}\right). 
\end{multline*}
Setting $I=\mathrm{Ran}(\sigma)$ and simplifying, we see that the above expression coincides with the right-hand side of \eqref{eq:symmetrization-G'}. 
\end{proof}

\section{Degenerations}\label{sc:degeneration}
The aim of this section is to indicate certain degeneration of the above results as parameters $(q,s)$ tend to certain special values. All of these degenerations can and should be approached independently, and we hope to return to them in a future work.

\subsection{Hall-Littlewood symmetric polynomials}\label{ss:HL} The Hall-Littlewood symmetric polynomials are very well studied, and we refer to \cite[Chapter III]{Macdonald1995} for definitions and notations. 

In order to reach these polynomials from our definitions, it suffices to set $s=0$. Then the symmetrization formulas of Theorem \ref{th:symmetrization} imply that for any $M,N,n\ge 0$, $\mu=0^{m_0}1^{m_1}2^{m_2}\cdots\in\s_M^+$, $\nu=0^{n_0}1^{n_1}2^{n_2}\cdots\in\s_n^+$, we have
\begin{gather*}
F_\mu(u_1,\dots,u_M)=\prod_{i\ge 0}{(q;q)_{m_i}}\cdot P_{\mu}(u_1,\dots,u_M)={(q;q)_{m_0}}\cdot Q_{\mu}(u_1,\dots,u_n),\\
G_\nu(v_1,\dots,v_N)=\prod_{i\ge 1}{(q;q)_{n_i}}\cdot P_{\nu}(v_1,\dots,v_N)=Q_{\nu}(v_1,\dots,v_N),
\end{gather*}
with $P_\la$'s and $Q_\la$'s as in \cite[\S III.2]{Macdonald1995}, and the parameter $t$ of \cite{Macdonald1995} is identified with our $q$. 

Recalling Definition \ref{df:conjugated}, we can also compare the skew functions \cite[Ch. III, (5.11'), (5.14)]{Macdonald1995} and Definitions \ref{df:F}, \ref{df:G} to conclude that
\begin{align*}
F_{\lambda/\mu}^c(u)&=P_{\lambda/\mu}(u), \qquad \la\in\s_L^+,\quad \mu\in\s_{L+1}^+,\quad u\in \C,\\
G_{\la/\nu}(v)&=Q_{\la/\nu}(v), \qquad \la\in\s_L^+,\quad \mu\in\s_{L}^+,\quad v\in \C.
\end{align*}

Since the branching relations of Proposition \ref{pr:branching} are exactly the same as for the Hall-Littlewood polynomials, cf. \cite[Ch. III, (5.5), (5.6)]{Macdonald1995}, the formulas for $F_\mu,G_\nu$ above also follow from the formulas for the skew functions. 

The skew Cauchy identity of Theorem \ref{th:skew-cauchy-single} and Corollary \ref{cr:skew-cauchy} now matches the corresponding identity for the Hall-Littlewood polynomials, see \cite[Ch. VI, Ex. 7.6]{Macdonald1995} where this identity is stated in the more general context of Macdonald polynomials. The Cauchy identity of Corollary \ref{cr:cauchy} turns into \cite[Ch. III, (4.4)]{Macdonald1995}. 

The Pieri type rules of Corollary \ref{cr:pieri}, after decomposing the product-eigenvalue in the left-hand sides according to the Cauchy identity and comparing the same degree coefficients of both sides, coincide with \cite[Ch. III, (5.7), (5.7')]{Macdonald1995}.

The spatial orthogonality of Theorem \ref{th:spatial} is \cite[Ch. VI, (9.5)-(9.6)]{Macdonald1995}, where again the statement in \cite{Macdonald1995} is in the more general context of Macdonald polynomials. 

The principal specializations of $F_\mu$ and $G_\nu$ as in Proposition \ref{pr:principal} correspond to \cite[Ch. III, Ex. 2.1]{Macdonald1995}. I do not know however if analogs 
of Theorem \ref{th:skew-R} and Corollary \ref{cr:skew-G} that describe principal specialization of the skew functions $G_{\nu/\mu}$, have been considered in the Hall-Littlewood context. 

\subsection{Inhomogeneous Hall-Littlewood polynomials}\label{ss:iHL} Instead of simply setting $s=0$ as we did in Section \ref{ss:HL}, let us send $s\to 0$ but also simultaneously scale the variables $u_i=sz_i$, $v_i=sw_i$. Then the weights $w_u(i_1,j_1;i_2,j_2)$ of Definition \ref{df:weights} divided by $s^{j_1}$ turn into
\begin{align*}
\widetilde w_z(m,0,m,0)&=1, \qquad
\widetilde w_z(m,1,m,1)={z-q^m}, \\
\widetilde w_z(m+1,0,m,1)&={z},\qquad 
\widetilde w_z(m,1,m+1,0)={1-q^{m+1}},
\end{align*}
which implies, via Definitions \ref{df:F}, \ref{df:G}, that there exist limits
$$
\widetilde F_{\la/\mu}(z_1,\dots,z_m)=\lim_{s\to 0} F_{\la/\mu}(sz_1,\dots,sz_m),\qquad
\widetilde G_{\la/\nu}(w_1,\dots,w_n)=\lim_{s\to 0} G_{\la/\mu}(sw_1,\dots,sw_n),
$$
and they are (inhomogeneous) polynomials whose top homogeneous coefficients coincide with the corresponding Hall-Littlewood versions of $F_{\la/\mu}$ and $G_{\la/\mu}$ from the previous section. 

Taking the same limit in the symmetrization formulas of Theorem \ref{th:symmetrization}, we read (using the notations of Theorem \ref{th:symmetrization})
\begin{align*}
\widetilde F_\mu(z_1,\dots,z_M)&={(1-q)^M}\,\sum_{\sigma\in S_M}\sigma\left(\prod_{1\le i<j\le M}\frac{z_i-qz_j}{z_i-z_j}\cdot\prod_{i=1}^M \left({z_i-1}\right)^{\mu_i}\right),
\\
\widetilde G_\nu(w_1,\dots,w_N)&=\frac{(1-q)^N}{(q;q)_{N-n+k}}\sum_{\sigma\in S_N}\sigma\left(\prod_{1\le i<j\le N} \frac{w_i-qw_j}{w_i-w_j}\cdot \prod_{i=1}^{n-k} {w_i}\left({w_i-1}\right)^{\nu_i-1}\right).
\end{align*}

All the results we proved for $F_{\la/\mu}$ and $G_{\la/\mu}$ carry over to $\widetilde F_{\la/\mu}$ and $\widetilde G_{\la/\mu}$, and to my best knowledge, none of them have appeared in the literature before, with the exception of the orthogonality relations and Remark \ref{rm:q-hahn}(iii), whose analogs were proved in \cite{BCPS1}.

\subsection{The Schur like case: $q=0$}\label{ss:schur}  The Schur symmetric polynomials, see e.g. \cite[Chapter I]{Macdonald1995} can be thought of as specializations of the Hall-Littlewood symmetric polynomials with the parameter set to 0. Accordingly, we can set $q=0$ in our definitions of $F$- and $G$-functions. The vertex weights of Definition \ref{df:weights} then take the form
\begin{align*}
w_u^{(q=0)}(m,0,m,0)&=\frac{1-\mathbf{1}_{m=0}\cdot su}{1-su},\qquad 
w_u^{(q=0)}(m,1,m,1)=\frac{u-\mathbf{1}_{m=0}\cdot s}{1-su}, \\
w_u^{(q=0)}(m+1,0,m,1)&=\frac{\left(1-\mathbf{1}_{m=0}\cdot s^2\right)u}{1-su},\qquad
w_u^{(q=0)}(m,1,m+1,0)=\frac{1}{1-su}.
\end{align*}
The symmetrization formulas of Theorem \ref{th:symmetrization} take the form of ratios of two determinants:
\begin{align*}
F_\mu^{(q=0)}(u_1,\dots,u_M)&=\frac{\det\left[\dfrac{u_i^{M-j}}{1-su_i}\left(\dfrac{u_i-s}{1-su_i}\right)^{\mu_j}\right]_{i,j=1}^M}{\prod_{1\le i<j\le M}(u_i-u_j)}\,,\\
G_\nu^{(q=0)}(v_1,\dots,v_N)
&=(1-s^2)^{\mathbf{1}_{k=0}}\frac{\det\left[\dfrac{v_i^{N-j}}{1-sv_i}\left(\dfrac{v_i}{v_i-s}\right)^{\mathbf{1}_{\nu_j>0}}\left(\dfrac{v_i-s}{1-sv_i}\right)^{\nu_j}\right]_{i,j=1}^N}{\prod_{1\le i<j\le N}(v_i-v_j)}\,.
\end{align*}
At $s=0$ both formulas turn into the celebrated formula for the Schur polynomials as a ratio of two alternants \cite[Ch. I, (3.1)]{Macdonald1995}. 

One easily sees that for $q=0$, the conjugation of Definition \ref{df:conjugated} leaves all skew $G$-functions unaffected, which means that it can be removed from all the Cauchy and Pieri type formulas of Section \ref{sc:cauchy}. 

As in the previous section, all our results carry over to this degenerate case, and it remains unclear to me whether any of them have been considered before. 

\subsection{Inhomogeneous Schur polynomials}\label{ss:ischur} By combining the degeneration procedures of Sections \ref{ss:iHL} and \ref{ss:schur}, i.e. taking $q=0$, $s\to 0$, and scaling the variables $u_i=sz_i$, $v_i=sw_i$, we observe the vertex weights
\begin{align*}
\widetilde w_z^{(q=0)}(m,0,m,0)&=1,\qquad 
\widetilde w_z^{(q=0)}(m,1,m,1)={z-\mathbf{1}_{m=0}}, \\
\widetilde w_z^{(q=0)}(m+1,0,m,1)&=z,\qquad
\widetilde w_z^{(q=0)}(m,1,m+1,0)=1,
\end{align*}
and symmetrization formulas
\begin{align*}
\widetilde F_\mu^{(q=0)}(z_1,\dots,z_M)&=\frac{\det\left[{z_i^{M-j}}\left({z_i-1}\right)^{\mu_j}\right]_{i,j=1}^M}{\prod_{1\le i<j\le M}(z_i-z_j)}\,,\\
\widetilde G_\nu^{(q=0)}(w_1,\dots,w_N)
&=\frac{\det\left[{w_i^{N-j}}\left(\dfrac{w_i}{w_i-1}\right)^{\mathbf{1}_{\nu_j>0}}\left({w_i-1}\right)^{\nu_j}\right]_{i,j=1}^N}{\prod_{1\le i<j\le N}(w_i-w_j)}\,.
\end{align*}

These are inhomogeneous polynomials whose top homogeneous components coincide with the Schur polynomials. 

The polynomials $\widetilde F_\mu^{(q=0)}(z_1,\dots,z_M)$ bear a certain similarity to the so-called Grothendieck polynomials as presented in \cite{MotegiSakai1}, \cite{MotegiSakai2}, see also references therein, as well as \cite{LS}, \cite{Lenart} for much earlier works on those polynomials. However, on the surface the connection to integrable lattice models, skew functions, and the Cauchy formulas for these two objects appear to be different. 
It would be very interesting to establish a direct link. 

As in Sections \ref{ss:iHL}, \ref{ss:schur}, our results carry over to this case as well, and we have so far been unable to find them in the literature. 

\subsection{Trigonometric to rational limit: $q\to 1$}\label{ss:XXX} The limit we consider here is equivalent to the well-known transition from the XXZ to the XXX model. We take
$$
q=\exp(\epsilon),\quad s=\exp(\epsilon\zeta), \quad u_i=\exp(\epsilon x_i),\quad v_i=\exp(\epsilon y_i), \qquad \epsilon\to 0. 
$$ 
Such limit of the vertex weights of Definition \ref{df:weights} gives
\begin{align*}
w_x^{\mathrm{(rational)}}(m,0,m,0)&=\frac{m+\zeta+x}{\zeta+x}, \qquad
w_x^{\mathrm{(rational)}}(m,1,m,1)=\frac{m+\zeta-x}{\zeta+x}, \\
w_x^{\mathrm{(rational)}}(m+1,0,m,1)&=\frac{m+2\zeta}{\zeta+x},\qquad
w_x^{\mathrm{(rational)}}(m,1,m+1,0)=\frac{m+1}{\zeta+x}.
\end{align*}
Taking the limit of symmetrization formulas of Theorem \ref{th:symmetrization} yields
\begin{align*}
F_\mu^{\mathrm{(rational)}}(x_1,\dots,x_M)&=\frac{1}{\prod_{i=1}^M (\zeta+x_i)}\,\sum_{\sigma\in S_M}\sigma\left(\prod_{1\le i<j\le M}\frac{x_i-x_j-1}{x_i-x_j}\cdot\prod_{i=1}^M \left(\frac{\zeta-x_i}{\zeta+x_i}\right)^{\mu_i}\right)\,,\\
G_\nu^{\mathrm{(rational)}}(y_1,\dots,y_N)&=\frac{(2\zeta)_n}{(N-n+k)!(2\zeta)_k}\\ \times \sum_{\sigma\in S_N}\sigma\Biggl(\prod_{1\le i<j\le N} &\frac{y_i-y_j-1}{y_i-y_j}\cdot \prod_{i=1}^{n-k} \frac{1}{(\zeta+y_i)(\zeta-y_i)}\left(\frac{\zeta-y_i}{\zeta+y_i}\right)^{\nu_i}\cdot \prod_{j=n-k+1}^N\frac{k+\zeta+y_i}{\zeta+y_i}\Biggr),
\end{align*}
with the Pochhammer notation $(a)_m=a(a+1)\cdots(a+m-1)$ for $m\ge 1$, and 1 for $m=0$. 

Once again, our results also have such limits, and we have not seen those in the literature.

\end{document}